\documentclass{amsart}

\usepackage[T1]{fontenc}
\usepackage{enumerate, amsmath, amsfonts, amssymb, amsthm, mathrsfs, wasysym, graphics, graphicx, xcolor, url, hyperref, hypcap,  shuffle, xargs, multicol, overpic, pdflscape, multirow, hvfloat, minibox, accents, array, xifthen, a4wide, ae, aecompl}
\usepackage{marginnote}
\hypersetup{colorlinks=true, citecolor=darkblue, linkcolor=darkblue}
\usepackage[all]{xy}
\usepackage[bottom]{footmisc}
\usepackage{tikz}
\usepackage{tkz-graph}
\usetikzlibrary{trees, decorations, decorations.markings, shapes, arrows, matrix, calc, fit, intersections, patterns}
\graphicspath{{figures/}}
\usepackage{caption}
\title[Geometric realizations of the accordion complex of a dissection]{Geometric realizations of the \\ accordion complex of a dissection}

\thanks{Partially supported by the French ANR grant SC3A~(15\,CE40\,0004\,01).}

\author{Thibault Manneville}
\address[Thibault Manneville]{LIX, \'Ecole Polytechnique}
\email{thibault.manneville@lix.polytechnique.fr}
\urladdr{\url{http://www.lix.polytechnique.fr/~manneville/}}

\author{Vincent Pilaud}
\address[Vincent Pilaud]{CNRS \& LIX, \'Ecole Polytechnique, Palaiseau}
\email{vincent.pilaud@lix.polytechnique.fr}
\urladdr{\url{http://www.lix.polytechnique.fr/~pilaud/}}

\newtheorem{theorem}{Theorem}
\newtheorem{corollary}[theorem]{Corollary}
\newtheorem{proposition}[theorem]{Proposition}
\newtheorem{lemma}[theorem]{Lemma}

\theoremstyle{definition}

\newtheorem{example}[theorem]{Example}
\newtheorem{remark}[theorem]{Remark}

\newcommand{\R}{\mathbb{R}} 
\newcommand{\Z}{\mathbb{Z}} 
\newcommand{\fS}{\mathfrak{S}} 
\newcommand{\cA}{\mathcal{A}} 
\newcommand{\rmX}{{\rm X}} 
\renewcommand{\b}[1]{\mathbf{#1}} 

\newcommand{\set}[2]{\left\{ #1 \;\middle|\; #2 \right\}} 
\newcommand{\bigset}[2]{\big\{ #1 \;\big|\; #2 \big\}} 
\newcommand{\ssm}{\smallsetminus} 
\newcommand{\dotprod}[2]{\left\langle \, #1 \; \middle| \; #2 \, \right\rangle} 
\newcommand{\symdif}{\,\triangle\,} 
\newcommand{\one}{{1\!\!1}} 
\newcommand{\eqdef}{\mbox{\,\raisebox{0.2ex}{\scriptsize\ensuremath{\mathrm:}}\ensuremath{=}\,}} 
\newcommand{\defeq}{\mbox{~\ensuremath{=}\raisebox{0.2ex}{\scriptsize\ensuremath{\mathrm:}} }} 
\newcommand{\transpose}[1]{{#1}^t} 

\newcommand{\Asso}{\mathsf{Asso}} 
\newcommand{\Acco}{\mathsf{Acco}} 
\newcommand{\Perm}{\mathsf{Perm}} 
\newcommand{\Para}{\mathsf{Para}} 
\newcommand{\Zono}{\mathsf{Zono}} 
\DeclareMathOperator{\face}{\mathbf{F}} 
\newcommand{\Fan}{\mathcal{F}} 
\newcommand{\Cone}{\mathrm{C}} 

\DeclareMathOperator{\conv}{conv} 

\newcommand{\fref}[1]{Figure~\ref{#1}} 
\newcommand{\ie}{\textit{i.e.}~} 
\newcommand{\eg}{\textit{e.g.}~} 
\definecolor{darkblue}{rgb}{0,0,0.7} 
\newcommand{\darkblue}{\color{darkblue}} 
\newcommand{\defn}[1]{\textsl{\darkblue #1}} 

\usepackage{todonotes}

\newcommand{\accordionComplex}{\mathcal{AC}} 
\newcommand{\accordionLattice}{\mathcal{AL}} 
\newcommand{\accordionFlipGraph}{\mathcal{AFG}} 
\newcommand{\polygon}{\mathrm{P}} 
\newcommand{\triangulation}{\mathrm{T}} 
\newcommand{\fanTriangulation}{\mathrm{F}} 
\newcommand{\quadrangulation}{\mathrm{Q}} 
\newcommand{\dissection}{\mathrm{D}} 
\newcommand{\cell}{\mathrm{C}} 
\newcommand{\quadrilateral}{\mathrm{Q}} 
\newcommand{\accordion}{\mathrm{A}} 
\newcommand{\zigzag}{\mathrm{Z}} 
\newcommand{\snake}{\reflectbox{$\mathrm{Z}$}} 
\newcommand{\signcirc}[3]{{\varepsilon_\circ \big( {#1} \in {#2} \;|\; {#3} \big)}} 
\newcommand{\signbullet}[3]{{\varepsilon_\bullet \big( {#1} \;|\; {#3} \in {#2} \big)}} 
\newcommand{\SSS}{\reflectbox{$\mathsf{Z}$}} 
\newcommand{\ZZZ}{\mathsf{Z}} 
\newcommand{\VVV}{{\mathsf{{V \hspace{-.1686cm} I\,}}}} 
\newcommand{\signature}{\varepsilon} 
\newcommand{\gvector}[2]{\mathbf{g}(#1 \,|\, #2)} 
\newcommand{\biggvector}[2]{\mathbf{g} \big( #1 \,|\, #2 \big)} 
\newcommand{\gvectors}[2]{\mathbf{g}(#1 \,|\, #2)} 
\newcommand{\biggvectors}[2]{\mathbf{g} \big( #1 \,|\, #2 \big)} 
\newcommand{\gvectorFan}{\mathcal{F}^\mathbf{g}} 

\newcommand{\comp}[2]{(#1 \,|\, #2)} 
\newcommand{\dvector}[2]{\mathbf{d}(#1 \,|\, #2)} 
\newcommand{\bigdvector}[2]{\mathbf{d} \big( #1  \,|\, #2 \big)} 
\newcommand{\dvectors}[2]{\mathbf{d}(#1 \,|\, #2)} 
\newcommand{\bigdvectors}[2]{\mathbf{d} \big( #1  \,|\, #2 \big)} 
\newcommand{\dvectorFan}{\mathcal{F}^\mathbf{d}} 
\newcommand{\cvector}[3]{\mathbf{c}(#1  \,|\, #3 \in #2)} 
\newcommand{\bigcvector}[3]{\mathbf{c} \big( #1  \,|\, #3 \in #2 \big)} 
\newcommand{\cvectors}[2]{\mathbf{c}(#1  \,|\, #2)} 
\newcommand{\bigcvectors}[2]{\mathbf{c} \big( #1  \,|\, #2 \big)} 
\newcommand{\allcvectors}[1]{\mathbf{C}(#1)} 
\newcommand{\cvectorFan}{\mathcal{F}^\mathbf{c}} 
\newcommand{\rhs}[2]{\omega(#1 \,|\, #2)} 
\newcommand{\bigrhs}[2]{\omega \big( #1  \,|\, #2 \big)} 
\newcommand{\point}[2]{\mathbf{p}(#1  \,|\, #2)} 
\newcommand{\bigpoint}[2]{\mathbf{p} \big( #1  \,|\, #2 \big)} 
\newcommand{\ray}{\mathbf{r}} 
\newcommand{\rays}{\mathbf{R}} 
\newcommand{\hs}{\mathbf{H}^{\le}} 
\newcommand{\HS}[2]{\mathbf{H}^{\le}(#1  \,|\, #2)} 
\newcommand{\bigHS}[2]{\mathbf{H}^{\le} \big( #1  \,|\, #2 \big)} 
\newcommand{\hyp}{\mathbf{H}^{=}} 
\newcommand{\Hyp}[2]{\mathbf{H}^{=}(#1 \,|\, #2 )} 
\newcommand{\bigHyp}[2]{\mathbf{H}^{=} \big( #1  \,|\, #2 \big)} 
\newcommand{\fix}[1]{\mathrm{Fix}(#1)} 
\newcommand{\quiver}{\mathrm{Q}} 
\newcommand{\CoxeterGroup}{\mathrm{W}} 
\newcommand{\mi}{-} 
\newcommand{\ma}{+} 
\newcommand{\ini}{\mathrm{ini}} 
\newcommand{\ex}{\mathrm{ex}} 
\newcommand{\projection}{\pi} 
\newcommand{\restrictedComplex}[3]{\Delta^{\b{#1}}(#2,#3)} 
\renewcommand{\restriction}[2]{\left.\kern-\nulldelimiterspace #1 \vphantom{\big|} \right|_{#2}}

\begin{document}

\begin{abstract}
Consider $2n$ points on the unit circle and a reference dissection~$\dissection_\circ$ of the convex hull of the odd points. The accordion complex of~$\dissection_\circ$ is the simplicial complex of non-crossing subsets of the diagonals with even endpoints that cross a connected subset of diagonals of~$\dissection_\circ$. In particular, this complex is an associahedron when~$\dissection_\circ$ is a triangulation and a Stokes complex when~$\dissection_\circ$ is a quadrangulation. In this paper, we provide geometric realizations (by polytopes and fans) of the accordion complex of any reference dissection~$\dissection_\circ$, generalizing known constructions arising from cluster algebras.

\medskip
\noindent
\textsc{keywords.} Permutahedra $\cdot$ Zonotopes $\cdot$ Associahedra $\cdot$ $\b{g}$-, $\b{c}$- and $\b{d}$-vectors.
\end{abstract}

\vspace*{.1cm}

\maketitle


The $(n-3)$-dimensional \defn{associahedron} is a simple polytope whose face poset is isomorphic to the reverse inclusion poset of non-crossing subsets of diagonals of a convex $n$-gon. Introduced in early works of D.~Tamari~\cite{Tamari} and \mbox{J.~Stasheff~\cite{Stasheff}}, it was first realized as a convex polytope by M.~Haiman~\cite{Haiman} and C.~Lee~\cite{Lee}, and later constructed by more systematic methods developed by several authors, in particular~\cite{GelfandKapranovZelevinsky, Loday, HohlwegLange, CeballosSantosZiegler}. Various relevant generalizations of the associahedron were introduced and studied, in particular secondary polytopes and fiber polytopes~\cite{GelfandKapranovZelevinsky, BilleraFillimanSturmfels}, generalized associahedra~\cite{FominZelevinsky-YSystems, ChapotonFominZelevinsky, HohlwegLangeThomas, Stella, Hohlweg} in connection to cluster algebras~\cite{FominZelevinsky-ClusterAlgebrasI, FominZelevinsky-ClusterAlgebrasII}, graph associahedra~\cite{CarrDevadoss, Postnikov, FeichtnerSturmfels, Zelevinsky, Pilaud-signedTreeAssociahedra, MannevillePilaud-compatibilityFans}, or brick polytopes~\cite{PilaudSantos-brickPolytope, PilaudStump-brickPolytope}.

In a different context, Y.~Baryshnikov~\cite{Baryshnikov} introduced the simplicial complex of crossing-free subsets of the set of diagonals of a polygon that are in some sense compatible with a reference quadrangulation~$\quadrangulation_\circ$. Although the precise definition of compatibility is a bit technical in~\cite{Baryshnikov}, it turns out that a diagonal is compatible with~$\quadrangulation_\circ$ if and only if it crosses a connected subset of diagonals of~$\quadrangulation_\circ$ that we call \defn{accordion} of~$\quadrangulation_\circ$. We thus call Y.~Baryshnikov's simplicial complex the \defn{accordion complex}~$\accordionComplex(\quadrangulation_\circ)$. A polytopal realization of~$\accordionComplex(\quadrangulation_\circ)$ was announced in~\cite{Baryshnikov}, but the explicit construction and its proof were never published as far as we know. Revisiting some combinatorial and algebraic properties of~$\accordionComplex(\quadrangulation_\circ)$, F.~Chapoton~\cite[Intro.\,p.4]{Chapoton-quadrangulations} raised three explicit challenges: first prove that the oriented dual graph of~$\accordionComplex(\quadrangulation_\circ)$ has a lattice structure extending the Tamari and Cambrian lattices~\cite{TamariFestschrift, Reading-CambrianLattices}; second construct geometric realizations of~$\accordionComplex(\quadrangulation_\circ)$ as fans and polytopes generalizing the known constructions of the associahedron; third show that the facets of~$\accordionComplex(\quadrangulation_\circ)$ are in bijection with other combinatorial objects \mbox{called serpent nests}~\cite[Sect.~4]{Chapoton-quadrangulations}.

In~\cite{GarverMcConville}, A.~Garver and T.~McConville defined and studied the accordion complex~$\accordionComplex(\dissection_\circ)$ of any reference dissection~$\dissection_\circ$ (their presentation slightly differs as they use a compatibility condition on the dual tree of the dissection~$\dissection_\circ$, but the simplicial complex is the same). In this context, they settled F.~Chapoton's lattice question, using lattice quotients of a lattice of biclosed sets. In this paper, we present geometric realizations of~$\accordionComplex(\dissection_\circ)$ for any reference dissection~$\dissection_\circ$, providing in particular an answer to F.~Chapoton's geometric question. In fact, we present three methods to realize~$\accordionComplex(\dissection_\circ)$ based on constructions of the classical associahedron.

Our first method is based on the $\b{g}$-vector fan. It belongs to a series of constructions of the (generalized) associahedra initiated by S.~Shnider and S.~Sternberg~\cite{ShniderSternberg}, popularised by \mbox{J.-L.~Loday~\cite{Loday}}, developed by C.~Hohlweg, C.~Lange and H.~Thomas~\cite{HohlwegLange, HohlwegLangeThomas} using works of N.~Reading and D.~Speyer~\cite{Reading-CambrianLattices, Reading-sortableElements, ReadingSpeyer}, and revisited by S.~Stella~\cite{Stella} and by V.~Pilaud, F.~Santos, and C.~Stump~\cite{PilaudSantos-brickPolytope, PilaudStump-brickPolytope}. It was recently extended by C.~Hohlweg, V.~Pilaud, and S.~Stella~\cite{HohlwegPilaudStella} to construct an associahedron parametrized by any initial triangulation. Here, we first extend to the $\dissection_\circ$-accordion complex~$\accordionComplex(\dissection_\circ)$ the $\b{g}$-vectors and $\b{c}$-vectors defined in the context of cluster algebras by S.~Fomin and A.~Zelevinski~\cite{FominZelevinsky-ClusterAlgebrasIV}. Note that $\b{c}$-vectors were already implicitly considered in~\cite{GarverMcConville}, while $\b{g}$-vectors are new in this context. When~$\dissection_\circ$ is a triangulation, our definitions coincide with those given in terms of triangulations and laminations for cluster algebras from surfaces by S.~Fomin and D.~Thurston~\cite{FominThurston}. We then show that the $\b{g}$-vectors with respect to the dissection~$\dissection_\circ$ support a complete simplicial fan~$\gvectorFan(\dissection_\circ)$ realizing the $\dissection_\circ$-accordion complex~$\accordionComplex(\dissection_\circ)$. Finally, we construct a $\dissection_\circ$-accordiohedron~$\Acco(\dissection_\circ)$ realizing the $\b{g}$-vector fan~$\gvectorFan(\dissection_\circ)$ by deleting inequalities from the facet description of the \mbox{$\dissection_\circ$-zonotope}~$\Zono(\dissection_\circ)$ obtained as the Minkowski sum of all $\b{c}$-vectors. See \fref{fig:exmAccordiohedra} for an illustration of $\dissection_\circ$-accordiohedra.

Our second method is based on the $\b{d}$-vector fan. This construction is inspired from the original cluster fan of S.~Fomin and A.~Zelevinsky~\cite{FominZelevinsky-ClusterAlgebrasII} later realized as a polytope by F.~Chapoton, S.~Fomin and A.~Zelevinsky~\cite{ChapotonFominZelevinsky}, and from the generalization of C.~Ceballos, F.~Santos and G.~Ziegler~\cite{CeballosSantosZiegler} to construct a compatibility fan and an associahedron from any initial triangulation. For any reference dissection~$\dissection_\circ$, we associate to each diagonal a $\b{d}$-vector which records the crossings of this diagonal with those of~$\dissection_\circ$. We show that the $\b{d}$-vectors support a complete simplicial fan realizing the $\dissection_\circ$-accordion complex~$\accordionComplex(\dissection_\circ)$ if and only if $\dissection_\circ$ contains no even interior cell. The polytopality of the resulting fan remains open in general, but was shown for arbitrary triangulations in~\cite{CeballosSantosZiegler}.

Finally, our third method is based on projections of associahedra. Namely, for any dissection~$\dissection_\circ$ and triangulation~$\triangulation_\circ$ such that~$\dissection_\circ \subseteq \triangulation_\circ$, the accordion complex~$\accordionComplex(\dissection_\circ)$ is a subcomplex of the simplicial associahedron~$\accordionComplex(\triangulation_\circ)$. It turns out that the $\b{g}$-vector fan~$\gvectorFan(\dissection_\circ)$ is then a section of the $\b{g}$-vector fan~$\gvectorFan(\triangulation_\circ)$ by a coordinate subspace. Therefore, the accordion complex~$\accordionComplex(\dissection_\circ)$ is realized by a projection of the associahedron~$\Asso(\triangulation_\circ)$ of~\cite{HohlwegPilaudStella}. This point of view provides a complementary perspective on accordion complexes that leads on the one hand to more concise but less instructive proofs of combinatorial and geometric properties of the accordion complex (pseudomanifold, $\b{g}$-vector fan, accordiohedron), and on the other hand to natural extensions to coordinate sections of the $\b{g}$-vector fan in arbitrary cluster algebras.

As recently observed in~\cite{GarverMcConville, PaluPilaudPlamondon, PilaudPlamondonStella, BrustleDouvilleMousavandThomasYildirim}, accordion complexes are prototypes of support $\tau$-tilting complexes introduced in~\cite{AdachiIyamaReiten}, for certain associative algebras called gentle algebras. In this context, $\b{g}$-vectors have a deep algebraic meaning and still define a $\b{g}$-vector fan. Although this fan is still polytopal for finite support $\tau$-tilting complexes, it is not in general obtained by deleting inequalities in the facet description of a zonotope. We refer to~\cite[Part~4]{PaluPilaudPlamondon} for details.

The paper is organized as follows. Section~\ref{sec:accordionComplex} introduces the accordion complex and accordion lattice of a dissection~$\dissection_\circ$. We essentially follow the definitions and arguments of A.~Garver and T.~McConville~\cite{GarverMcConville}, except that we prefer to work on the dissection~$\dissection_\circ$ rather than on its dual graph. Section~\ref{sec:gvectorFan} is devoted to the generalization of the $\b{g}$-vector fan and the associahedra of~\cite{HohlwegLange, HohlwegPilaudStella}. Section~\ref{sec:dvectorFan} discusses the generalization of the construction of the $\b{d}$-vector fan and associahedra of~\cite{FominZelevinsky-ClusterAlgebrasII, CeballosSantosZiegler}. Finally, Section~\ref{sec:projection} shows that the accordion complex is realized by a projection of a well-chosen associahedron and presents related questions on cluster algebras, subcomplexes of the cluster complex, and sections of the $\b{g}$-vector~fan.


\section{The accordion complex and the accordion lattice}
\label{sec:accordionComplex}

In this section, we define the accordion complex~$\accordionComplex(\dissection_\circ)$ of a dissection~$\dissection_\circ$, show that it is a pseudomanifold, and define an orientation of its dual graph. Our definitions and proofs are essentially translations of the arguments of A.~Garver and T.~McConville~\cite{GarverMcConville} given in terms of the dual tree of the dissection~$\dissection_\circ$. However our presentation in terms of dissections is more convenient for our purposes.


\subsection{The accordion complex}

\begin{figure}[t]
	\capstart
	\centerline{\includegraphics[scale=.8]{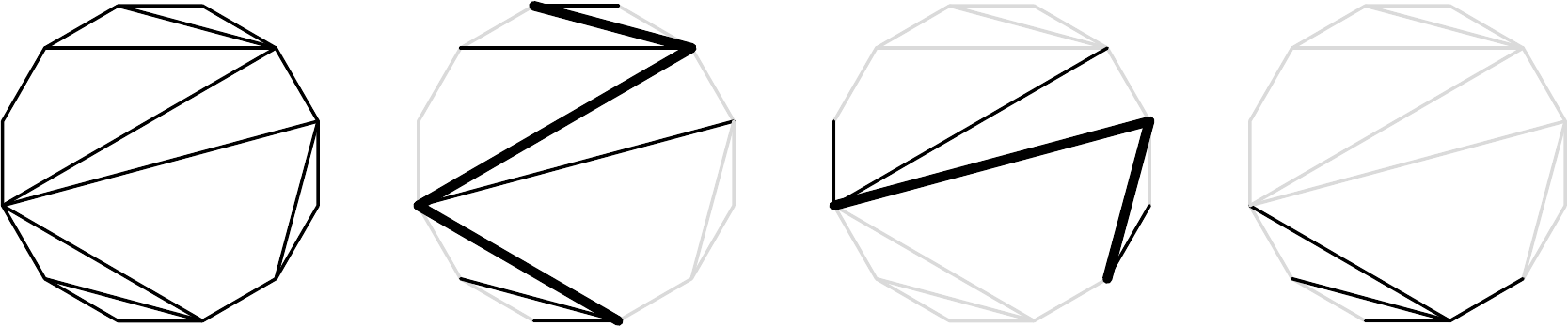}}
	\caption{A dissection~$\dissection$ (left) and three accordions whose zigzags are bolded (middle and right).}
	\label{fig:exmAccordions}
\end{figure}

Let~$\polygon$ be a convex polygon. We call \defn{diagonals} of~$\polygon$ the segments connecting two vertices of~$\polygon$. This includes both the internal diagonals and the external diagonals (or boundary edges) of~$\polygon$. A \defn{dissection} of~$\polygon$ is a set~$\dissection$ of non-crossing internal diagonals of~$\polygon$. The \defn{cells} of~$\dissection$ are the closures of the connected components of~$\polygon$ minus the diagonals of~$\dissection$. A \defn{triangulation} (resp.~\defn{quadrangulation}) is a dissection whose cells are all triangles (resp.~quadrangles).

We denote by~$\overline{\dissection}$ the dissection~$\dissection$ together with all boundary edges of~$\polygon$. A \defn{cut} of~$\dissection$ is the subset of~$\overline{\dissection}$ intersected by a line crossing two boundary edges of~$\polygon$. An \defn{accordion} is a connected cut. By definition, an accordion is a tree and contains two boundary edges of~$\polygon$. The \defn{zigzag} of an accordion~$\accordion$ is the chain obtained by deleting all degree~$1$ vertices of~$\accordion$. A \defn{subaccordion} of~$\dissection$ is a connected subset of~$\dissection$ intersected by a segment in the interior of~$\polygon$. Note that any subaccordion of an accordion~$\accordion$ consists of the diagonals of~$\accordion$ between two internal diagonals of~$\accordion$. Note that we include boundary edges of~$\polygon$ in the accordions of~$\dissection$, but not in the subaccordions nor in the zigzags of~$\dissection$. See \fref{fig:exmAccordions}.

We consider $2n$ points on the unit circle labeled clockwise by \mbox{$1_\circ$, $2_\bullet$, $3_\circ$, $4_\bullet$, \dots, $(2n-1)_\circ$, $(2n)_\bullet$}. We say that~$1_\circ, \dots, (2n-1)_\circ$ are the \defn{hollow vertices} while~$2_\bullet, \dots, (2n)_\bullet$ are the \defn{solid vertices}. The \defn{hollow polygon} is the convex hull~$\polygon_\circ$ of~$1_\circ, \dots, (2n-1)_\circ$ while the \defn{solid polygon} is the convex hull~$\polygon_\bullet$ of~$2_\bullet, \dots, (2n)_\bullet$. We simultaneously consider \defn{hollow diagonals}~$\delta_\circ$ (with two hollow vertices) and \defn{solid diagonals}~$\delta_\bullet$ (with two solid vertices), but we never consider diagonals with one hollow vertex and one solid vertex. Similarly, we consider \defn{hollow dissections}~$\dissection_\circ$ (of the hollow polygon, with only hollow diagonals) and \defn{solid dissections}~$\dissection_\bullet$ (of the solid polygon, with only solid diagonals), but never mix hollow and solid diagonals in a dissection. To help distinguish them, hollow (resp.~solid) vertices and diagonals appear red (resp.~blue) in all pictures.

We fix an arbitrary reference hollow dissection~$\dissection_\circ$. A solid diagonal~$\delta_\bullet$ is a \defn{$\dissection_\circ$-accordion diagonal} if the hollow diagonals of~$\overline{\dissection}_\circ$ crossed by~$\delta_\bullet$ form an accordion of~$\dissection_\circ$. In other words, $\delta_\bullet$ cannot enter and exit a cell of~$\dissection_\circ$ using two non-incident diagonals. For example, note that for any hollow diagonal~$i_\circ j_\circ \in \overline{\dissection}_\circ$, the solid diagonals~$(i-1)_\bullet (j-1)_\bullet$ and~$(i+1)_\bullet (j+1)_\bullet$ are $\dissection_\circ$-accordion diagonals (here and throughout, labels are considered modulo~$2n$). In particular, all boundary edges of the solid polygon are $\dissection_\circ$-accordion diagonals. A \defn{$\dissection_\circ$-accordion dissection} is a set of non-crossing internal $\dissection_\circ$-accordion diagonals. We define the \defn{$\dissection_\circ$-accordion complex} to be the simplicial complex~$\accordionComplex(\dissection_\circ)$ of $\dissection_\circ$-accordion dissections.

\begin{example}
As a running example, we consider the reference dissection~$\dissection_\circ^\ex$ of \fref{fig:exmAccordionDissections}\,(left). Examples of maximal $\dissection_\circ^\ex$-accordion dissections are given in \fref{fig:exmAccordionDissections}\,(right). The $\dissection_\circ^\ex$-accordion complex is illustrated in \fref{fig:exmAccordionComplex}\,(left).

\begin{figure}[h]
	\capstart
	\centerline{\includegraphics[scale=.8]{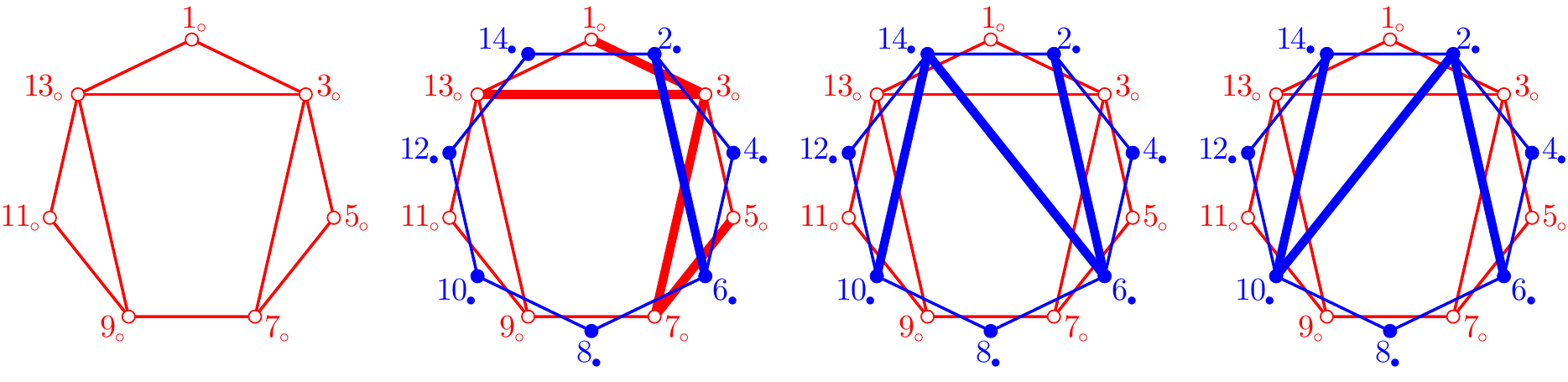}}
	\caption{A hollow dissection~$\dissection_\circ^\ex$, a solid $\dissection_\circ^\ex$-accordion diagonal whose corresponding hollow accordion is bolded, and two maximal solid $\dissection_\circ^\ex$-accordion dissections.}
	\label{fig:exmAccordionDissections}
\end{figure}

\begin{figure}
	\capstart
	\centerline{\includegraphics[scale=1]{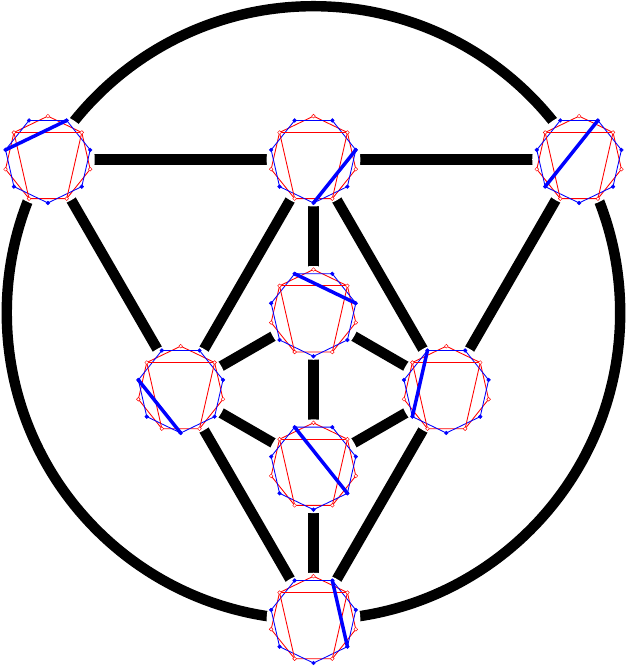} \qquad \includegraphics[scale=1]{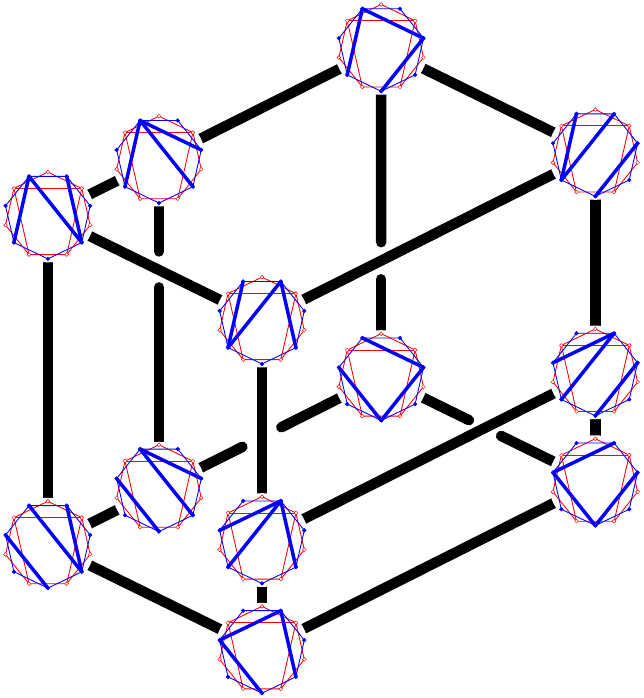}}
	\caption{The $\dissection_\circ^\ex$-accordion complex (left) and the $\dissection_\circ^\ex$-accordion lattice (right), oriented from bottom to top, for the reference hollow dissection~$\dissection_\circ^\ex$ of \fref{fig:exmAccordionDissections}\,(left).}
	\label{fig:exmAccordionComplex}
\end{figure}
\end{example}

\begin{example}
\label{exm:specialReferenceDissections}
Special reference hollow dissections~$\dissection_\circ$ give rise to special accordion complexes~$\accordionComplex(\dissection_\circ)$:
\begin{itemize}
\item If~$\dissection_\circ$ is the empty dissection with the whole hollow polygon as unique cell, then the \mbox{$\dissection_\circ$-accor}\-dion complex~$\accordionComplex(\dissection_\circ)$ is reduced to the empty $\dissection_\circ$-accordion dissection.
\item If~$\dissection_\circ$ has a unique internal diagonal, then the $\dissection_\circ$-accordion complex~$\accordionComplex(\dissection_\circ)$ consists of only two points.
\item For a hollow triangulation~$\triangulation_\circ$, all solid diagonals are $\triangulation_\circ$-accordions, so that the \mbox{$\triangulation_\circ$-accor}\-dion complex~$\accordionComplex(\triangulation_\circ)$ is the simplicial associahedron.
\item For a hollow quadrangulation~$\quadrangulation_\circ$, a solid diagonal is a $\quadrangulation_\circ$-accordion if and only if it does not cross two opposite edges of a quadrangle of~$\quadrangulation_\circ$. The $\quadrangulation_\circ$-accordion complex~$\accordionComplex(\quadrangulation_\circ)$ is thus the Stokes complex defined by Y.~Baryshnikov~\cite{Baryshnikov} and studied \mbox{by F.~Chapoton~\cite{Chapoton-quadrangulations}}.
\end{itemize}
\end{example}

\begin{remark}
Following the original definition of the non-crossing complex of A.~Garver and T.~McConville~\cite{GarverMcConville}, the accordion complex could equivalently be defined in terms of the dual tree~$\dissection_\circ^\star$ of~$\dissection_\circ$ (with one node in each cell of~$\dissection$ and one edge connecting two adjacent cells). 
More precisely, the duality provides the following dictionary between the two definitions:

\medskip
\centerline{
\begin{tabular}{ccc}
present paper & & A.~Garver and T.~McConville~\cite{GarverMcConville} \\
\hline
reference dissection~$\dissection_\circ$ & $\longleftrightarrow$ & embedded tree~$\dissection_\circ^\star$ \\
diagonal~$u_\bullet v_\bullet$ of~$\polygon_\bullet$ & $\longleftrightarrow$ & path connecting the leaves~$u_\bullet^\star$ and~$v_\bullet^\star$ of~$\dissection_\circ^\star$ \\
$\dissection_\circ$-accordion diagonal & $\longleftrightarrow$ & arc (path where any two consecutive edges belong to the \\ & & boundary of a face of the complement of~$\dissection_\circ^\star$ in the unit disk) \\
$\dissection_\circ$-subaccordion & $\longleftrightarrow$ & segment \\
$\dissection_\circ$-accordion complex & $\longleftrightarrow$ & non-crossing complex of~$\dissection_\circ^\star$
\end{tabular}
}

\medskip\noindent
The $\b{g}$-, $\b{c}$- and $\b{d}$-vectors defined in Section~\ref{subsec:gcvectors} could as well be defined in terms of~$\dissection_\circ^\star$. In fact, $\b{c}$-vectors were already implicitly considered in~\cite{GarverMcConville}, while $\b{g}$- and $\b{d}$-vectors are new in this context. For this paper, we find more convenient to work directly~with~dissections, in particular in Sections~\ref{sec:dvectorFan} and~\ref{sec:projection}.
\end{remark}


\subsection{Two structural observations}
\label{subsec:structuralObservations}

Before studying the accordion complex in details in Section~\ref{subsec:pseudomanifold}, we present two simple structural observations. For this, let us recall two classical notions on simplicial complexes. The \defn{join} of two simplicial complexes~$\Delta,\Delta'$ with disjoint ground sets~$X,X'$ is the simplicial complex~$\Delta * \Delta'$ with ground set~$X \sqcup X'$ whose faces are disjoint unions of faces of~$\Delta$ with faces of~$\Delta'$. For a face~$\dissection$ in a simplicial complex~$\Delta$ on~$X$, the \defn{link} of~$\dissection$ is the simplicial complex on~$X \ssm \dissection$ whose faces are the subsets~$\dissection'$ of~$X \ssm \dissection$ such that~$\dissection \cup \dissection'$ is a face of~$\Delta$.

\begin{proposition}
\label{prop:reduction}
If the reference hollow dissection~$\dissection_\circ$ has a cell containing~$p$ boundary edges of the hollow polygon~$\polygon_\circ$, then the $\dissection_\circ$-accordion complex~$\accordionComplex(\dissection_\circ)$ is the join of $p$ accordion complexes.
\end{proposition}

\begin{proof}
Assume that~$\dissection_\circ$ has a cell~$\cell_\circ$ containing~$p$ boundary edges of the hollow polygon~$\polygon_\circ$. Let~$\cell_\circ^1, \dots, \cell_\circ^p$ denote the $p$ (possibly empty) connected components of the hollow polygon minus~$\cell_\circ$. For~$i \in [p] \eqdef \{1, \dots, p\}$, let~$\dissection_\circ^i$ denote the dissection formed by the cell~$\cell_\circ$ together with the cells of~$\dissection_\circ$ contained in the closure of~$\cell_\circ^i$. Observe that for~$i \ne j$, the internal diagonals of~$\dissection_\circ^i$ are not incident to the internal diagonals of~$\dissection_\circ^j$. Thus, no $\dissection_\circ$-accordion can contain internal diagonals from distinct dissections~$\dissection_\circ^i$ and~$\dissection_\circ^j$. Therefore, the set of $\dissection_\circ$-accordion diagonals is the union of the sets of $\dissection_\circ^i$-accordion diagonals for~$i \in [p]$. Moreover, for~$i \ne j$, the $\dissection_\circ^i$-accordion diagonals do not cross the $\dissection_\circ^j$-accordion diagonals. It follows that the $\dissection_\circ$-accordion complex is the join of the $\dissection_\circ^i$-accordion complexes:~${\accordionComplex(\dissection_\circ) = \accordionComplex(\dissection_\circ^1) * \dots * \accordionComplex(\dissection_\circ^p)}$.
\end{proof}

\begin{remark}
\label{rem:reduction}
In view of Proposition~\ref{prop:reduction}, we can do the following reductions:
\begin{enumerate}[(i)]
\item If a non-triangular cell of~$\dissection_\circ$ has two consecutive boundary edges~$\gamma_\circ, \delta_\circ$ of the hollow polygon, then contracting~$\gamma_\circ$ and~$\delta_\circ$ to a single boundary edge preserves the $\dissection_\circ$-accordion complex.
\item If a cell of~$\dissection_\circ$ has two non-consecutive boundary edges of the hollow polygon, then the $\dissection_\circ$-accordion complex is a join of smaller accordion complexes.
\end{enumerate}
In all the examples of the paper, we therefore only consider dissections where any non-triangular cell of~$\dissection_\circ$ has at most one boundary edge. All of our constructions work in general, but are just obtained as products or joins of the non-degenerate situation.
\end{remark}

\begin{proposition}
\label{prop:links}
The links in an accordion complex are joins of accordion complexes.
\end{proposition}

\begin{proof}
Consider a $\dissection_\circ$-accordion dissection~$\dissection_\bullet$ with cells~$\cell_\bullet^1, \dots, \cell_\bullet^p$. Let~$\dissection_\circ^i$ denote the hollow dissection obtained from~$\dissection_\circ$ by contracting all hollow boundary edges which do not cross~$\cell_\bullet^i$. Then a diagonal~$\delta_\bullet$ of a cell~$\cell_\bullet^i$ is a $\dissection_\circ$-accordion diagonal if and only if it is a~$\dissection_\circ^i$-accordion diagonal. Moreover, for~$i \ne j$, the diagonals of~$\cell_\bullet^i$ do not cross the diagonals of~$\cell_\bullet^j$. It follows that the link of~$\dissection_\bullet$ in~$\accordionComplex(\dissection_\circ)$ is isomorphic to the join~$\accordionComplex(\dissection_\circ^1) * \dots * \accordionComplex(\dissection_\circ^p)$.
\end{proof}


\subsection{Pseudo-manifold}
\label{subsec:pseudomanifold}

We now prove that the accordion complex~$\accordionComplex(\dissection_\circ)$ is a \defn{pseudomanifold}, \ie that it is:
\begin{enumerate}[(i)]
\item \defn{pure}: all maximal $\dissection_\circ$-accordion dissections have the same number of diagonals as~$\dissection_\circ$, and
\item \defn{thin}: any codimension~$1$ simplex of~$\accordionComplex(\dissection_\circ)$ is contained in exactly two maximal $\dissection_\circ$-accordion dissections.
\end{enumerate}
We follow the arguments of A.~Garver and T.~McConville~\cite{GarverMcConville} (except that they work on the dual tree of the dissection~$\dissection_\circ$). A much more concise but less instructive proof of the pseudomanifold property will be derived from geometric considerations in Remark~\ref{rem:simplerProofs}.

Recall that we denote by~$\overline{\dissection}_\circ$ the set formed by~$\dissection_\circ$ together with all boundary edges of the hollow polygon. An \defn{angle}~$u_\circ v_\circ w_\circ$ of~$\overline{\dissection}_\circ$ is a pair~$\{u_\circ v_\circ, v_\circ w_\circ\}$ of two consecutive diagonals of~$\overline{\dissection}_\circ$ around a common vertex~$v_\circ$, called the \defn{apex}. Note that~$\overline{\dissection}_\circ$ has~${2|\dissection_\circ|+n = 2|\overline{\dissection}_\circ|-n}$ angles. Observe also that an accordion~$\accordion_\circ$ of~$\dissection_\circ$ can be seen as a sequence of $|\accordion_\circ|-1$ angles where two consecutive angles are separated by a diagonal of~$\accordion_\circ$. We say that a solid vertex~$p_\bullet$ belongs to an angle~$u_\circ v_\circ w_\circ$ if it lies in the cone generated by the edges~$v_\circ u_\circ$ and~$v_\circ w_\circ$ of the angle. The main observation is given in the following statement.

\begin{lemma}
\label{lem:sameAngle}
Let~$\dissection_\bullet$ be a maximal $\dissection_\circ$-accordion dissection, and let~$p_\bullet, q_\bullet, r_\bullet, s_\bullet$ denote four consecutive vertices of a cell~$\cell_\bullet$ of~$\dissection_\bullet$ (with possibly~$p_\bullet = s_\bullet$ if~$\cell_\bullet$ is a triangle). Then~$p_\bullet$ and~$s_\bullet$ belong to the same angle of the accordion of~$\overline{\dissection}_\circ$ which is crossed by~$q_\bullet r_\bullet$.
\end{lemma}

\begin{proof}
Let~$\accordion_\circ$ be the accordion of~$\overline{\dissection}_\circ$ which is crossed by~$q_\bullet r_\bullet$. Assume that $p_\bullet$ and~$s_\bullet$ belong to distinct angles of~$\accordion_\circ$. Then they are separated by a diagonal~$\varepsilon_\circ$ of~$\accordion_\circ$. Therefore, there are two boundary edges~$q_\bullet r_\bullet$ and~$u_\bullet v_\bullet$ of~$\cell_\bullet$ with distinct vertices such that the hollow diagonal~$\varepsilon_\circ$ separates the vertices~$q_\bullet, u_\bullet$ from the vertices~$r_\bullet, v_\bullet$. Let~$\gamma_\circ^1, \dots, \gamma_\circ^i = \varepsilon_\circ, \dots, \gamma_\circ^a$ (resp.~${\delta_\circ^1, \dots, \delta_\circ^j = \varepsilon_\circ, \dots, \delta_\circ^b}$) denote the diagonals of~$\dissection_\circ$ crossed by~$q_\bullet r_\bullet$ from~$q_\bullet$ to~$r_\bullet$ (resp.~ crossed by~$u_\bullet v_\bullet$ from~$u_\bullet$ to~$v_\bullet$). Then the hollow diagonals~$\gamma_\circ^1, \dots, \gamma_\circ^i = \varepsilon_\circ = \delta_\circ^j, \dots, \delta_\circ^b$ which are crossed by~$q_\bullet v_\bullet$ also form an accordion. It follows that~$\dissection_\bullet$ is not maximal as we can still include~$q_\bullet v_\bullet$.
\end{proof}

Consider now an angle~$u_\circ v_\circ w_\circ$ of~$\overline{\dissection}_\circ$. In any maximal $\dissection_\circ$-accordion dissection~$\dissection_\bullet$, the set~$\rmX_\bullet$ of diagonals of~$\overline{\dissection}_\bullet$ that cross both $u_\circ v_\circ$ and~$v_\circ w_\circ$ is non-empty (since it contains the boundary edge~$(v-1)_\bullet (v+1)_\bullet$) and totally ordered (since the diagonals of~$\dissection_\bullet$ do not cross). Let~$\delta_\bullet$ be the largest diagonal of~$\rmX_\bullet$ (meaning the farthest from~$v_\circ$). We say that the diagonal~$\delta_\bullet$ \defn{closes} the angle~$u_\circ v_\circ w_\circ$. Note that each angle of~$\overline{\dissection}_\circ$ is closed by precisely one diagonal~of~$\overline{\dissection}_\bullet$. The following lemma is stated and proved in~\cite{GarverMcConville} in terms of the dual tree~$\dissection_\circ^\star$ of the dissection~$\dissection_\circ$.

\begin{lemma}[\cite{GarverMcConville}]
\label{lem:closeTwoAngles}
For any maximal $\dissection_\circ$-accordion dissection~$\dissection_\bullet$, each internal diagonal~$\delta_\bullet$ of~$\dissection_\bullet$ closes two angles of~$\overline{\dissection}_\circ$ (one apex on each side of~$\delta_\bullet$) while each boundary edge of the solid polygon closes one angle of~$\overline{\dissection}_\circ$. Therefore the accordion complex~$\accordionComplex(\dissection_\circ)$ is pure of dimension~$|\dissection_\circ|$.
\end{lemma}

\begin{proof}
The first sentence is a consequence of Lemma~\ref{lem:sameAngle}: for any four consecutive vertices~$p_\bullet, q_\bullet, r_\bullet, s_\bullet$ of a cell of~$\overline{\dissection}_\bullet$, the diagonal~$q_\bullet r_\bullet$ closes the unique angle of the accordion of~$\overline{\dissection}_\circ$ crossed by~$q_\bullet r_\bullet$ that contains the vertices~$p_\bullet$ and~$s_\bullet$. Therefore, $q_\bullet r_\bullet$ closes precisely two angles (resp.~one angle) of~$\dissection_\circ$ if it is an internal diagonal (resp.~a boundary edge of the solid polygon). We finally obtain by double-counting that $2|\dissection_\circ|+n = |\{\text{angles of } \overline{\dissection}_\circ\}| = 2|\dissection_\bullet|+n$ and thus~$|\dissection_\bullet| = |\dissection_\circ|$ for any maximal $\dissection_\circ$-accordion dissection~$\dissection_\bullet$.
\end{proof}

We are now ready to prove that the $\dissection_\circ$-accordion complex is thin, \ie that each internal diagonal of a maximal $\dissection_\circ$-accordion dissection can be flipped into a unique other internal diagonal to form a new maximal $\dissection_\circ$-accordion dissection. Here and throughout the paper, $X \symdif Y$ denotes the symmetric difference of two sets~$X,Y$ defined by~${X \symdif Y \eqdef (X \ssm Y) \cup (Y \ssm X)}$.

The following notations are illustrated in \fref{fig:exmFlip}. 
Let~$\dissection_\bullet$ be a maximal $\dissection_\circ$-accordion dissection and~$\delta_\bullet$ be a diagonal of~$\dissection_\bullet$. Let~$u_\circ$ and~$v_\circ$ be the apices of the angles of~$\dissection_\circ$ closed by~$\delta_\bullet$, let~$\mu_\bullet$ and~$\nu_\bullet$ denote the edges of the cells of~$\dissection_\bullet$ containing~$\delta_\bullet$, which separate~$\delta_\bullet$ from~$u_\circ$ and~$v_\circ$ respectively, and let~$\quadrilateral_\bullet$ denote the quadrilateral defined by the four vertices of~$\mu_\bullet$ and~$\nu_\bullet$. Note that $\delta_\bullet$ is a diagonal of~$\quadrilateral_\bullet$, and let~$\delta_\bullet'$ denote the other diagonal. 

\begin{figure}
	\capstart
	\centerline{\includegraphics[scale=1]{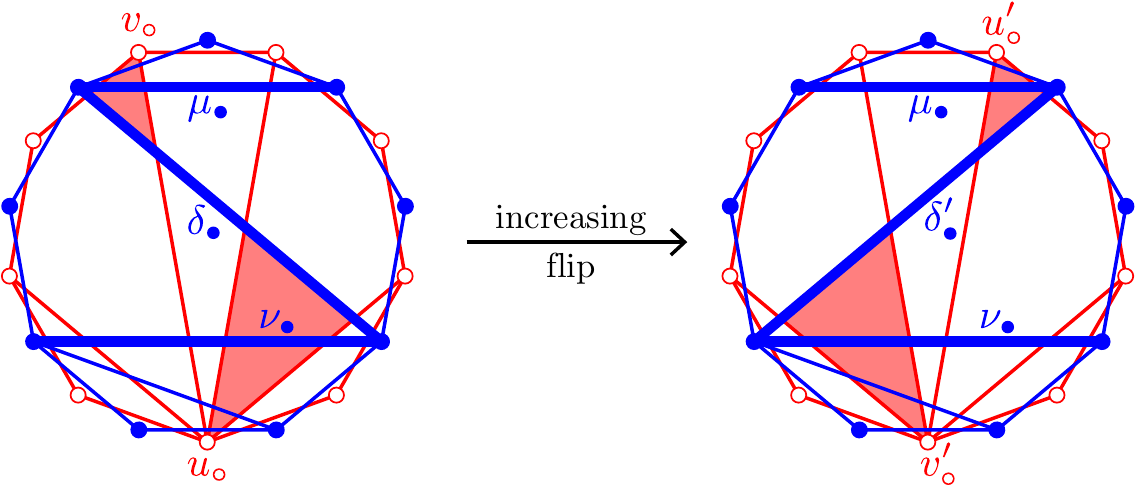}}
	\caption{Two maximal $\dissection_\circ$-accordion dissection~$\dissection_\bullet$ (left) and~$\dissection_\bullet'$ (right) related by the flip of~$\delta_\bullet$ to~$\delta_\bullet'$. The angles of~$\dissection_\circ$ closed by~$\delta_\bullet$ and~$\delta_\bullet'$ are shaded. The flip is oriented from~$\dissection_\bullet$ to~$\dissection_\bullet'$.}
	\label{fig:exmFlip}
\end{figure}

\begin{lemma}[\cite{GarverMcConville}]
\label{lem:flip}
With the previous notations, the collection of diagonals~$\dissection_\bullet' \eqdef \dissection_\bullet \symdif \{\delta_\bullet, \delta_\bullet'\}$ is a maximal $\dissection_\circ$-accordion dissection, and~$\dissection_\bullet$ and~$\dissection_\bullet'$ are the only maximal $\dissection_\circ$-accordion dissections containing~$\dissection_\bullet \ssm \{\delta_\bullet\}$. In other words, the accordion complex~$\accordionComplex(\dissection_\circ)$ is thin.
\end{lemma}

\begin{proof}
We first observe that~$\delta_\bullet'$ is a $\dissection_\circ$-accordion diagonal, since the edges of~$\overline{\dissection}_\circ$ crossed by~$\delta_\bullet'$ are obtained by merging three subaccordions of~$\dissection_\circ$: the subaccordion formed by the diagonals of~$\overline{\dissection}_\circ$ crossed by~$\mu_\bullet$ but not~$\delta_\bullet$ nor~$\nu_\bullet$, the subaccordion formed by the diagonals of~$\overline{\dissection}_\circ$ crossed by~$\delta_\bullet$, $\mu_\bullet$ and~$\nu_\bullet$, and the subaccordion formed by the diagonals of~$\overline{\dissection}_\circ$ crossed by~$\nu_\bullet$ but not~$\delta_\bullet$ nor~$\mu_\bullet$. Moreover, $\delta_\bullet$ and $\delta_\bullet'$ are the only $\dissection_\circ$-accordion diagonals compatible with~$\dissection_\bullet \ssm \{\delta_\bullet\}$. Indeed, any other such diagonal would cross~$\delta_\bullet$ and~$\delta_\bullet'$ (by maximality of~$\dissection_\bullet$ and~$\dissection_\bullet'$), and thus also the subaccordion~$\accordion_\circ$ of~$\dissection_\circ$ crossed by~$\delta_\bullet$ and~$\delta_\bullet'$ (because it cannot cross~$\mu$ and~$\nu$). But it would then improperly intersect the two cells of~$\dissection_\circ$ containing precisely one diagonal of~$\accordion_\circ$.
\end{proof}

The \defn{$\dissection_\circ$-accordion flip graph} is the dual graph~$\accordionFlipGraph(\dissection_\circ)$ of the $\dissection_\circ$-accordion complex: its~vertices are the maximal $\dissection_\circ$-accordion dissections, and its edges are the \defn{flips} between them, \ie the pairs~$\{\dissection_\bullet, \dissection_\bullet'\}$ of maximal $\dissection_\circ$-accordion dissections with~$\dissection_\bullet \ssm \{\delta_\bullet\} = \dissection_\bullet' \ssm \{\delta_\bullet'\}$. See~\mbox{\fref{fig:exmAccordionComplex}\,(right)}.


\subsection{The accordion lattice}
\label{subsec:accordionLattice}

We now define a natural orientation on the $\dissection_\circ$-accordion flip graph. We use the same notations as in Lemma~\ref{lem:flip} (see also \fref{fig:exmFlip}), where~$\dissection_\bullet \ssm \{\delta_\bullet\} = \dissection_\bullet' \ssm \{\delta_\bullet'\}$ and~$\delta_\bullet, \delta_\bullet'$ are the two diagonals of the quadrilateral defined by~$\mu_\bullet, \nu_\bullet$. Observe that one of the path~$\mu_\bullet \delta_\bullet \nu_\bullet$ and~$\mu_\bullet \delta_\bullet' \nu_\bullet$ forms a~$\SSS$ while the other forms a~$\ZZZ$, see \fref{fig:exmFlip}. We then orient the flip from the dissection containing the~$\SSS$ to that containing the~$\ZZZ$. See \fref{fig:exmAccordionComplex}\,(right) for an illustration of $\dissection_\circ$-accordion oriented flip graph (where the graph is oriented from bottom to top).

A.~Garver and T.~McConville introduced a natural closure on sets of $\dissection_\circ$-subaccordions, and showed that the inclusion poset of biclosed sets of $\dissection_\circ$-subaccordions is a well-behaved lattice (namely, semidistributive, congruence-uniform and polygonal). Then, they introduced a lattice quotient map from biclosed sets of $\dissection_\circ$-subaccordions to maximal $\dissection_\circ$-accordion dissections, which imply the following statement.

\begin{theorem}[\cite{GarverMcConville}]
The $\dissection_\circ$-accordion oriented flip graph is the Hasse diagram of a lattice, that we call the \defn{$\dissection_\circ$-accordion lattice} and denote by~$\accordionLattice(\dissection_\circ)$.
\end{theorem}

In particular, the $\dissection_\circ$-accordion oriented flip graph is connected and acyclic, and has a unique source~$\dissection_\bullet^\mi \eqdef \set{(i-1)_\bullet (j-1)_\bullet}{i_\circ j_\circ \in \dissection_\circ}$ (obtained by slightly rotating~$\dissection_\circ$ counterclockwise) and a unique sink~$\dissection_\bullet^\ma \eqdef \set{(i+1)_\bullet (j+1)_\bullet}{i_\circ j_\circ \in \dissection_\circ}$ (obtained by slightly rotating~$\dissection_\circ$ clockwise).

\begin{example}
Following Example~\ref{exm:specialReferenceDissections}, note that special reference hollow dissections~$\dissection_\circ$ give rise to special accordion lattices~$\accordionLattice(\dissection_\circ)$, as it was already observed in~\cite{GarverMcConville}:
\begin{itemize}
\item For a fan triangulation~$\fanTriangulation_\circ$ (\ie where all internal diagonals are incident to a common vertex), the $\fanTriangulation_\circ$-accordion lattice~$\accordionLattice(\fanTriangulation_\circ)$ is the famous Tamari lattice~\cite{Tamari, TamariFestschrift} defined equivalently by slope increasing flips on triangulations of a convex polygon, by right rotations on binary trees, or by flips on Dyck paths.
\item In general, accordion lattices of accordion triangulations (\ie with no interior triangle) precisely correspond to type~$A$ Cambrian lattices defined by N.~Reading~\cite{Reading-CambrianLattices}.
\item For an arbitrary triangulation~$\triangulation_\circ$ (with or without interior triangle), the $\triangulation_\circ$-accordion oriented flip graph~$\accordionFlipGraph(\accordion_\circ)$ is a particular instance of the oriented exchange graphs of $2$-acyclic quivers defined by T.~Br\"ustle, G.~Dupont and M.~P\'erotin~\cite{BrustleDupontPerotin}. These oriented exchange graphs are far more general and their transitive closures are in general not lattices.
\item For a quadrangulation~$\quadrangulation_\circ$, the $\quadrangulation_\circ$-accordion lattice~$\accordionLattice(\quadrangulation_\circ)$ is the Stokes poset on \mbox{$\quadrangulation_\circ$-com}\-patible quadrangulations studied by F.Chapoton~\cite{Chapoton-quadrangulations}.
\end{itemize}
\end{example}

The following statement is a direct consequence of Proposition~\ref{prop:reduction}.

\begin{proposition}
If the reference hollow dissection~$\dissection_\circ$ has a cell containing~$p$ boundary edges of the hollow polygon~$\polygon_\circ$, then the $\dissection_\circ$-accordion lattice~$\accordionLattice(\dissection_\circ)$ is a Cartesian product of $p$ accordion lattices.
\end{proposition}

\begin{proof}
Consider the dissections~$\dissection_\circ^1, \dots, \dissection_\circ^p$ as in the proof of Proposition~\ref{prop:reduction}. Since any increasing flip in~$\accordionComplex(\dissection_\circ)$ is an increasing flip in one of the~$\accordionComplex(\dissection_\circ^i)$, we obtain that the $\dissection_\circ$-accordion lattice is the Cartesian product of the $\dissection_\circ^i$-accordion lattices:~${\accordionLattice(\dissection_\circ) = \accordionLattice(\dissection_\circ^1) \times \dots \times \accordionLattice(\dissection_\circ^p)}$.
\end{proof}

In particular, if two consecutive boundary edges~$\gamma_\circ, \delta_\circ$ of the hollow polygon belong to the same non-triangular cell of~$\dissection_\circ$, then contracting~$\gamma_\circ$ and~$\delta_\circ$ to a single boundary edge preserves the \mbox{$\dissection_\circ$-accor}\-dion lattice. This shows the following statement conjectured for quadrangulations in~\cite{Chapoton-quadrangulations} and proved in~\cite{BateniMannevillePilaud}.

\begin{corollary}
Consider an accordion dissection~$\accordion_\circ$, \ie a dissection where each cell has at most $2$ edges which are internal diagonals of~$\polygon_\circ$.
Then the $\accordion_\circ$-accordion lattice is a Cambrian lattice.
\end{corollary}

\begin{remark}
Call \defn{cell-sequence} of a dissection the sequence whose $i$th entry is its number \mbox{of~$(i+2)$-cells}. For example, the dissection of \fref{fig:exmAccordionDissections}\,(left) has cell-sequence~$3,1,0^\infty$ and all $(p+2)$-angulations of a $(pm+2)$-gon have cell-sequence~$0^{p-1},m,0^\infty$. Observe that the flip preserves the cell-sequence. Thus, all maximal $\dissection_\circ$-accordion dissections have the same cell-sequence~as~$\dissection_\circ$.
\end{remark}

We conclude this section with a reciprocity result on accordion dissections.

\begin{proposition}
\label{prop:reciprocity}
Let~$\dissection_\circ$ be a hollow dissection and~$\dissection_\bullet$ be a solid dissection. Then~$\dissection_\bullet$ is a maximal $\dissection_\circ$-accordion dissection if and only if~$\dissection_\circ$ is a maximal $\dissection_\bullet$-accordion dissection.
\end{proposition}

\begin{proof}
Since~$\dissection_\bullet^\mi \eqdef \set{(i-1)_\bullet (j-1)_\bullet}{i_\circ j_\circ \in \dissection_\circ}$ and~$\dissection_\bullet^\ma \eqdef \set{(i+1)_\bullet (j+1)_\bullet}{i_\circ j_\circ \in \dissection_\circ}$ are both \mbox{$\dissection_\circ$-ac}\-cordion dissections, we already know that~$\dissection_\circ$ is a \mbox{$\dissection_\bullet^\mi$-accor}\-dion dissection. Observe now in \fref{fig:exmFlip} that if~$\dissection_\bullet$ and~$\dissection_\bullet'$ are maximal $\dissection_\circ$-accordion dissections connected by a flip, then $\dissection_\circ$ is a $\dissection_\bullet$-accordion dissection if and only if it is a $\dissection_\bullet'$-accordion dissection. Indeed, if $\delta_\bullet$ belongs to the zigzag of the $\dissection_\bullet$-accordion~$\accordion_\bullet$ of a hollow diagonal~$\delta_\circ$, then~$\delta_\circ$ crosses both~$\mu_\bullet$ and~$\nu_\bullet$,  but then $\delta_\circ$ also crosses~$\delta_\bullet'$, and thus $\delta_\circ$ crosses the $\dissection_\bullet'$-accordion $\accordion_\bullet \symdif \{\delta_\bullet, \delta_\bullet'\}$. Since the $\dissection_\circ$-accordion flip graph is connected, we obtain that $\dissection_\circ$ is a $\dissection_\bullet$-accordion dissection for any maximal $\dissection_\circ$-accordion dissection~$\dissection_\bullet$. Finally, maximality follows since all maximal $\dissection_\circ$-accordion dissections have $|\dissection_\circ|$ diagonals. The equivalence follows by symmetry.
\end{proof}


\section{The $\b{g}$-vector fan}
\label{sec:gvectorFan}

In this Section, we construct accordiohedra using $\b{g}$- and $\b{c}$-vectors. Our construction is in the same spirit as the Cambrian fans of N.~Reading and D.~Speyer~\cite{Reading-CambrianLattices, Reading-sortableElements, ReadingSpeyer} and their polytopal realizations by C.~Hohlweg, C.~Lange and H.~Thomas~\cite{HohlwegLange, HohlwegLangeThomas}, recently extended in~\cite{HohlwegPilaudStella} to any initial triangulation, acyclic or not. A different approach to the $\b{g}$-vector fan together with an alternative polytopal realization will be presented in Section~\ref{sec:projection}.


\subsection{$\b{g}$- and $\b{c}$-vectors}
\label{subsec:gcvectors}

Consider a hollow dissection~$\dissection_\circ$ and a solid dissection~$\dissection_\bullet$ that are maximal accordion dissections of each other (see Proposition~\ref{prop:reciprocity}), and let~$\delta_\circ \in \dissection_\circ$ and~$\delta_\bullet \in \dissection_\bullet$. When~$\delta_\circ$ crosses~$\delta_\bullet$, we let~$\mu_\circ$ and~$\nu_\circ$ be the other diagonals of~$\overline{\dissection}_\circ$ crossed by~$\delta_\bullet$ in the two cells of~$\dissection_\circ$ containing~$\delta_\circ$. We say that~$\delta_\bullet$ \defn{slaloms} on~$\delta_\circ$ if~$\mu_\circ \delta_\circ \nu_\circ$ forms a path, and we define~$\signcirc{\delta_\circ}{\dissection_\circ}{\delta_\bullet}$ to be $1$, $-1$, or~$0$ depending on whether~$\mu_\circ \delta_\circ \nu_\circ$ forms a~$\ZZZ$, a~$\SSS$, or a~$\VVV$. Similarly we let~$\mu_\bullet$ and~$\nu_\bullet$ be the other diagonals of~$\overline{\dissection}_\bullet$ crossed by~$\delta_\circ$ in the two cells of~$\dissection_\bullet$ containing~$\delta_\bullet$, we say that~$\delta_\circ$ slaloms on~$\delta_\bullet$ if~$\mu_\bullet \delta_\bullet \nu_\bullet$ forms a path, and we define~$\signbullet{\delta_\circ}{\dissection_\bullet}{\delta_\bullet}$ to be $1$, $-1$, or~$0$ depending on whether~$\mu_\bullet \delta_\bullet \nu_\bullet$ forms a $\SSS$, a $\ZZZ$, or a~$\VVV$. Note that the sign convention for~$\signcirc{\delta_\circ}{\dissection_\circ}{\delta_\bullet}$ and~$\signbullet{\delta_\circ}{\dissection_\bullet}{\delta_\bullet}$ is opposite: the reciprocity already observed in Proposition~\ref{prop:reciprocity} naturally reverses the orientation. More informally, we exchange the role of hollow and solid dissections by looking at the picture from the opposite side of the blackboard, which of course reverses the orientation. Finally, if~$\delta_\circ$ and~$\delta_\bullet$ do not cross, then we let~$\signcirc{\delta_\circ}{\dissection_\circ}{\delta_\bullet} = \signbullet{\delta_\circ}{\dissection_\bullet}{\delta_\bullet} = 0$. Let~$(\b{e}_{\delta_\circ})_{\delta_\circ \in \dissection_\circ}$ denote the canonical basis of~$\R^{\dissection_\circ}$. As in~\cite{HohlwegPilaudStella}, we define the following vectors:

\begin{enumerate}[(i)]

\item the \defn{$\b{g}$-vector} of~$\delta_\bullet$ with respect to~$\dissection_\circ$ is
\({
\biggvector{\dissection_\circ}{\delta_\bullet} \eqdef \sum_{\delta_\circ \in \dissection_\circ} \signcirc{\delta_\circ}{\dissection_\circ}{\delta_\bullet} \, \b{e}_{\delta_\circ}.
}\)
We also define~$\biggvectors{\dissection_\circ}{\dissection_\bullet} \eqdef \bigset{\biggvector{\dissection_\circ}{\delta_\bullet}}{\delta_\bullet \in \dissection_\bullet}$. \\[-.2cm]

\item the \defn{$\b{c}$-vector} of~$\delta_\bullet \in \dissection_\bullet$ with respect to~$\dissection_\circ$ is
\({
\bigcvector{\dissection_\circ}{\dissection_\bullet}{\delta_\bullet} \eqdef \sum_{\delta_\circ \in \dissection_\circ} \signbullet{\delta_\circ}{\dissection_\bullet}{\delta_\bullet} \, \b{e}_{\delta_\circ}.
}\)
We denote by~$\bigcvectors{\dissection_\circ}{\dissection_\bullet} \eqdef \bigset{\bigcvector{\dissection_\circ}{\dissection_\bullet}{\delta_\bullet}}{\delta_\bullet \in \dissection_\bullet}$ the set of~$\b{c}$-vectors of the diagonals of~$\dissection_\bullet$  and by~$\allcvectors{\dissection_\circ} \eqdef \bigcup_{\dissection_\bullet} \bigcvectors{\dissection_\circ}{\dissection_\bullet}$ the set of all $\b{c}$-vectors with respect to~$\dissection_\circ$.

\end{enumerate}

\begin{example}
Consider the hollow dissection~$\dissection_\circ^\ex = \{3_\circ 7_\circ, 3_\circ 13_\circ, 9_\circ 13_\circ\}$ and the rightmost solid dissection~$\dissection_\bullet^\ex = \{2_\bullet 6_\bullet, 2_\bullet 10_\bullet, 10_\bullet 14_\bullet\}$ of \fref{fig:exmAccordionDissections}. Then we have for example
\begin{itemize}
\item $\signcirc{3_\circ 13_\circ}{\dissection_\circ^\ex}{2_\bullet 10_\bullet} = 1$ since the path~$1_\circ - 3_\circ - 13_\circ - 9_\circ$ forms a~$\ZZZ$, 
\item $\signcirc{9_\circ 13_\circ}{\dissection_\circ^\ex}{2_\bullet 10_\bullet} = -1$ since the path~$3_\circ - 13_\circ - 9_\circ - 11_\circ$ forms a~$\SSS$, and
\item $\signcirc{3_\circ 13_\circ}{\dissection_\circ^\ex}{2_\bullet 6_\bullet} = 0$ since~$3_\circ$ connects~$1_\circ, 13_\circ, 7_\circ$ as a~$\VVV$.
\end{itemize}
Moreover, we have
\[
\begin{array}{l@{\hspace{2cm}}l}
	\biggvector{\dissection_\circ^\ex}{2_\bullet 6_\bullet} = \b{e}_{3_\circ 7_\circ}, &
	\bigcvector{\dissection_\circ^\ex}{\dissection_\bullet^\ex}{2_\bullet 6_\bullet} = \b{e}_{3_\circ 7_\circ}, \\
	\biggvector{\dissection_\circ^\ex}{2_\bullet 10_\bullet} = \b{e}_{3_\circ 13_\circ} - \b{e}_{9_\circ 13_\circ}, &
	\bigcvector{\dissection_\circ^\ex}{\dissection_\bullet^\ex}{2_\bullet 10_\bullet} = \b{e}_{3_\circ 13_\circ}, \\
	\biggvector{\dissection_\circ^\ex}{10_\bullet 14_\bullet} = - \b{e}_{9_\circ 13_\circ}, &
	\bigcvector{\dissection_\circ^\ex}{\dissection_\bullet^\ex}{10_\bullet 14_\bullet} = - \b{e}_{3_\circ 13_\circ} - \b{e}_{9_\circ 13_\circ}.
\end{array}
\]
\end{example}

\begin{example}
\label{exm:gcvectors}
For any hollow diagonal~$i_\circ j_\circ \in \dissection_\circ$, we have
\[
\begin{array}{@{\hspace{-.2cm}}l@{\hspace{1.9cm}}l}
	\biggvector{\dissection_\circ}{(i-1)_\bullet (j-1)_\bullet} = - \b{e}_{i_\circ j_\circ}, &
	\bigcvector{\dissection_\circ}{\dissection_\bullet^\mi}{(i-1)_\bullet (j-1)_\bullet} = - \b{e}_{i_\circ j_\circ}, \\
	\biggvector{\dissection_\circ}{(i+1)_\bullet (j+1)_\bullet} =   \b{e}_{i_\circ j_\circ}, &
	\bigcvector{\dissection_\circ}{\dissection_\bullet^\ma}{(i+1)_\bullet (j+1)_\bullet} =   \b{e}_{i_\circ j_\circ}.
\end{array}
\]

\end{example}

\begin{remark}
For a hollow triangulation~$\triangulation_\circ$, our definitions of $\b{g}$- and $\b{c}$-vectors coincide with the shear coordinates of S.~Fomin and D.~Thurston~\cite{FominThurston}, defined in the much more general context of cluster algebras on surfaces~\cite{FominShapiroThurston}.
\end{remark}

\begin{remark}
\label{rem:quiver}
Consider the quiver~$\quiver(\dissection_\circ)$ of the reference dissection~$\dissection_\circ$, with one node on each internal diagonal of~$\dissection_\circ$ and one arrow between two diagonals counter-clockwise consecutive around a cell of~$\dissection_\circ$. Let~$\CoxeterGroup(\dissection_\circ)$ be the reflection group whose Dynkin diagram is the underlying graph of~$\quiver(\dissection_\circ)$. Then all $\b{g}$-vectors of the $\dissection_\circ$-accordion diagonals are weights of~$\CoxeterGroup(\dissection_\circ)$ and all $\b{c}$-vectors of~$\allcvectors{\dissection_\circ}$ are roots~of~$\CoxeterGroup(\dissection_\circ)$.
\end{remark}

\begin{remark}
\label{rem:informalgcvectors}
Informally, the $\b{g}$- and $\b{c}$-vectors can be interpreted as follows:
\begin{enumerate}[(i)]
\item The $\b{g}$-vector~$\gvector{\dissection_\circ}{\delta_\bullet}$ has coordinate $1$ and $-1$ alternating along the zigzag of the accordion crossed by~$\delta_\bullet$ in~$\dissection_\circ$, and coordinate~$0$ on all other diagonals of~$\dissection_\circ$.
\item The $\b{c}$-vector~$\cvector{\dissection_\circ}{\dissection_\bullet}{\delta_\bullet}$ is, up to a sign, the characteristic vector of the diagonals 
of the subaccordion of~$\dissection_\circ$ crossed by both diagonals~$\mu_\bullet$ and~$\nu_\bullet$ of Lemma~\ref{lem:flip} (see also \fref{fig:exmFlip}). Thus, any $\b{c}$-vector is either \defn{positive} (only non-negative coordinates) or \defn{negative} (only non-positive~coordinates).
\end{enumerate}
\end{remark}

In fact, the $\b{g}$-vectors are clearly in bijection with the accordions and with the zigzags in~$\dissection_\circ$. In contrast, many pairs~$(\delta_\bullet, \dissection_\bullet)$ produce the same $\b{c}$-vector ${\cvector{\dissection_\circ}{\dissection_\bullet}{\delta_\bullet}}$. For example, if two dissections $\dissection_\bullet, \dissection_\bullet'$ contain~$\delta_\bullet$ and have the same cells incident to~$\delta_\bullet$, then~${\cvector{\dissection_\circ}{\dissection_\bullet}{\delta_\bullet} = \cvector{\dissection_\circ}{\dissection_\bullet'}{\delta_\bullet}}$. The set of $\b{c}$-vectors~$\allcvectors{\dissection_\circ}$ without repetitions can be understood~as~follows.

\begin{lemma}
\label{lem:bijectioncvectors}
There are bijections between:
\begin{itemize}
\item the negative (resp.~positive) $\b{c}$-vectors of~$\allcvectors{\dissection_\circ}$,
\item the subaccordions of~$\dissection_\circ$,
\item the $\dissection_\circ$-accordion diagonals not in the source dissection~$\dissection_\bullet^\mi \eqdef \set{(i-1)_\bullet (j-1)_\bullet}{i_\circ j_\circ \in \dissection_\circ}$ (resp.~not in the sink dissection~$\dissection_\bullet^\ma \eqdef \set{(i+1)_\bullet (j+1)_\bullet}{i_\circ j_\circ \in \dissection_\circ}$).
\end{itemize}
\end{lemma}

\begin{proof}
By Remark~\ref{rem:informalgcvectors}\,(ii), the support of any $\b{c}$-vector is a subaccordion of~$\dissection_\circ$. Reciprocally, let~$\accordion_\circ$ be a subaccordion of~$\dissection_\circ$, let~$\cell_\circ$ and~$\cell_\circ'$ denote the two cells of~$\dissection_\circ$ containing exactly one diagonal of~$\accordion_\circ$, and let~$p_\circ, q_\circ, r_\circ, s_\circ$ (resp.~$p_\circ', q_\circ', r_\circ', s_\circ'$) denote the four consecutive vertices in clockwise order around~$\cell_\circ$ (resp.~around~$\cell_\circ'$) such that~$q_\circ r_\circ$ (resp.~$q_\circ' r_\circ'$) is the diagonal of~$\accordion_\circ$ in~$\cell_\circ$ (resp.~in~$\cell_\circ'$). Let~$\delta_\bullet \eqdef (s-1)_\bullet (s'-1)_\bullet$, ${\mu_\bullet \eqdef (p+1)_\bullet (s'-1)_\bullet}$ and~$\nu_\bullet \eqdef (p'+1)_\bullet (s-1)_\bullet$ and consider any $\dissection_\circ$-accordion dissection~$\dissection_\bullet$ containing~$\{\mu_\bullet, \delta_\bullet, \nu_\bullet\}$. Then~$\accordion_\circ$ is precisely the support of the negative $\b{c}$-vector~$\cvector{\dissection_\circ}{\dissection_\bullet}{\delta_\bullet}$. Finally, we have associated to the subaccordion~$\accordion_\circ$ of~$\dissection_\circ$ a $\dissection_\circ$-diagonal~$\delta_\bullet = (s-1)_\bullet (s'-1)_\bullet$ which cannot be in~$\dissection_\bullet^\mi$ as otherwise~$s_\circ s_\circ'$ would cross~$q_\circ r_\circ$. Reciprocally, $\accordion_\circ$ is precisely the set of diagonals of~$\dissection_\circ$ crossed by~$\delta_\bullet$ and not incident to~$s_\circ$ or~$s_\circ'$.
\end{proof}

The $\b{g}$-vectors and $\b{c}$-vectors are connected in the following two statements, inspired and motivated by classical analogues in cluster algebra theory.

\begin{proposition}
\label{prop:gvectorscvectorsDualBases}
For any maximal $\dissection_\circ$-accordion dissection~$\dissection_\bullet$, the set of $\b{g}$-vectors~$\gvectors{\dissection_\circ}{\dissection_\bullet}$ and the set of $\b{c}$-vectors~$\cvectors{\dissection_\circ}{\dissection_\bullet}$ form dual bases.
\end{proposition}

\begin{proof}
Let~$\dotprod{\cdot}{\cdot}$ denote the standard Euclidean inner product of~$\R^{\dissection_\circ}$.
Given two solid diagonals~$\gamma_\bullet, \delta_\bullet$ of~$\dissection_\bullet$, we want to compute~$\dotprod{\gvector{\dissection_\circ}{\gamma_\bullet}}{\cvector{\dissection_\circ}{\dissection_\bullet}{\delta_\bullet}}$. By Remark~\ref{rem:informalgcvectors}\,(i), the $\b{g}$-vector~$\gvector{\dissection_\circ}{\gamma_\bullet}$ has coordinate $\pm 1$ alternating along the zigzag~$\zigzag_\circ$ of the accordion crossed by~$\gamma_\bullet$ in~$\dissection_\circ$, and coordinate~$0$ on all other diagonals of~$\dissection_\circ$. Moreover, by Remark~\ref{rem:informalgcvectors}\,(ii), the $\b{c}$-vector~$\cvector{\dissection_\circ}{\dissection_\bullet}{\delta_\bullet}$ has coordinate~$\pm 1$ on the diagonals of~$\dissection_\circ$ which slalom on~$\delta_\bullet$ in~$\dissection_\bullet$, and coordinate~$0$ on all other diagonals of~$\dissection_\circ$. We thus need to understand how the diagonals of~$\zigzag_\circ$ slalom on~$\delta_\bullet$ in~$\dissection_\bullet$. See \fref{fig:exmProof} for a schematic illustration. Observe that there is an even (resp.~odd) number of hollow diagonals of~$\zigzag_\circ$ that slalom on~$\delta_\bullet$ when~$\delta_\bullet \ne \gamma_\bullet$ (resp.~when~$\delta_\bullet = \gamma_\bullet$). Moreover, since they are non-crossing, all hollow diagonals of~$\zigzag_\circ$ slaloming on~$\delta_\bullet$ do it the same way (either all as a~$\SSS$ or all as a~$\ZZZ$). Finally, when~$\gamma_\bullet = \delta_\bullet$, consider the first hollow diagonal~$\delta_\circ$ of the zigzag~$\zigzag_\circ$ which slaloms on~$\delta_\bullet$. Then~$\delta_\circ$ slaloms on~$\delta_\bullet$ in the opposite way as~$\delta_\bullet$ slaloms on~$\delta_\circ$. This shows that
\[
\dotprod{\biggvector{\dissection_\circ}{\gamma_\bullet}}{\bigcvector{\dissection_\circ}{\dissection_\bullet}{\delta_\bullet}} = \sum_{\delta_\circ \in \dissection_\circ} \signcirc{\delta_\circ}{\dissection_\circ}{\gamma_\bullet} \cdot \signbullet{\delta_\circ}{\dissection_\bullet}{\delta_\bullet} = \one_{\gamma = \delta},
\]
since we sum an even number of alternating~$\pm 1$ when~$\gamma_\bullet \ne \delta_\bullet$, and an odd number of alternating~$\pm 1$ starting by a~$1$ when~$\gamma_\bullet \ne \delta_\bullet$. In other words, $\gvectors{\dissection_\circ}{\dissection_\bullet}$ and~$\cvectors{\dissection_\circ}{\dissection_\bullet}$ form dual bases.
\begin{figure}[t]
	\capstart
	\centerline{\includegraphics[scale=1.5]{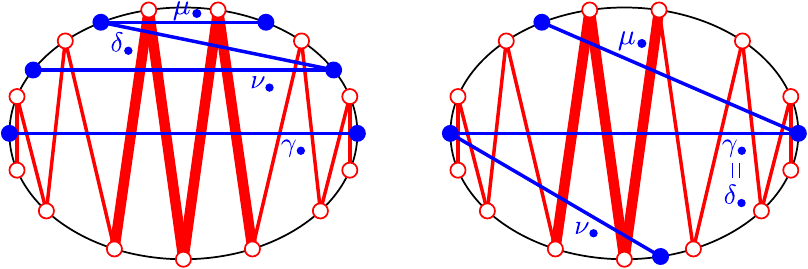}}
	\caption{Illustration of the proof of Proposition~\ref{prop:gvectorscvectorsDualBases}. The red hollow diagonals form the zigzag of~$\gamma_\bullet$, and the bolded ones are slaloming on~$\delta_\bullet$. There are an even number of bolded diagonals when~$\gamma_\bullet \ne \delta_\bullet$ (left) and an odd number when~$\gamma_\bullet = \delta_\bullet$ (right).}
	\label{fig:exmProof}
\end{figure}
\end{proof}

\begin{proposition}
\label{prop:reciprocitygcvectors}
Let~$\dissection_\circ$ be a hollow dissection and~$\dissection_\bullet$ be a solid dissection such that~$\dissection_\circ$ and~$\dissection_\bullet$ are maximal accordion dissections of each other (see Proposition~\ref{prop:reciprocity}). Then
\[
\biggvectors{\dissection_\circ}{\dissection_\bullet} = -\transpose{\bigcvectors{\dissection_\bullet}{\dissection_\circ}}
\qquad\text{and}\qquad
\bigcvectors{\dissection_\circ}{\dissection_\bullet} = -\transpose{\biggvectors{\dissection_\bullet}{\dissection_\circ}},
\]
where we consider the sets of $\b{g}$-vectors~$\gvectors{\dissection_\circ}{\dissection_\bullet}$ and $\b{c}$-vectors~$\cvectors{\dissection_\circ}{\dissection_\bullet}$ as matrices in~$\R^{\dissection_\circ \times \dissection_\bullet}$, and $\transpose{M}$ denotes the transpose of a matrix~$M$.
\end{proposition}

\begin{proof}
We immediately derive from the definitions that for any~$\delta_\circ \in \dissection_\circ$ and~$\delta_\bullet \in \dissection_\bullet$,
\[
\biggvectors{\dissection_\circ}{\dissection_\bullet}_{(\delta_\circ, \delta_\bullet)} = \signcirc{\delta_\circ}{\dissection_\circ}{\delta_\bullet} = -\signbullet{\delta_\bullet}{\dissection_\circ}{\delta_\circ} = -\bigcvectors{\dissection_\bullet}{\dissection_\circ}_{(\delta_\bullet, \delta_\circ)},
\]
which shows~$\gvectors{\dissection_\circ}{\dissection_\bullet} = -\transpose{\cvectors{\dissection_\bullet}{\dissection_\circ}}$. The other equality follows by exchanging~$\dissection_\circ$~and~$\dissection_\bullet$.
\end{proof}

\begin{corollary}
\label{coro:signCoherence}
For any maximal $\dissection_\circ$-accordion dissection~$\dissection_\bullet$, we have the following \defn{sign coherence}:
\begin{enumerate}[(i)]
\item for any~$\delta_\bullet \in \dissection_\bullet$, all coordinates of the $\b{c}$-vector~$\cvector{\dissection_\circ}{\dissection_\bullet}{\delta_\bullet}$ have the same sign,
\item for any~$\delta_\circ \in \dissection_\circ$, the $\delta_\circ$-coordinates of all $\b{g}$-vectors~$\gvector{\dissection_\circ}{\delta_\bullet}$ for~$\delta_\bullet \in \dissection_\bullet$ have the~same~sign.
\end{enumerate}
\end{corollary}

\begin{proof}
Point~(i) was already seen in Remark~\ref{rem:informalgcvectors}\,(ii), and Point~(ii) follows by Proposition~\ref{prop:reciprocitygcvectors}.
\end{proof}


\subsection{$\b{c}$-vector fan and $\dissection_\circ$-zonotope}

Define the \defn{$\b{c}$-vector fan} of~$\dissection_\circ$ to be the complete polyhedral fan~$\cvectorFan(\dissection_\circ)$ given by the arrangement of the linear hyperplanes orthogonal to the $\b{c}$-vectors of~$\allcvectors{\dissection_\circ}$. Be careful: in contrast to the $\b{g}$- and $\b{d}$-vector fans defined later, the $\b{c}$-vectors are not the rays of~$\cvectorFan(\dissection_\circ)$ but the normal vectors of the hyperplanes supporting the facets of~$\cvectorFan(\dissection_\circ)$.

We call \defn{$\dissection_\circ$-zonotope} the Minkowski sum~$\Zono(\dissection_\circ)$ of all $\b{c}$-vectors:
\[
\Zono(\dissection_\circ) \eqdef \sum_{\b{c} \in \allcvectors{\dissection_\circ}} \b{c}.
\]
The normal fan of the $\dissection_\circ$-zonotope~$\Zono(\dissection_\circ)$ is the $\b{c}$-vector fan~$\cvectorFan(\dissection_\circ)$. Note that the $\b{c}$-vector fan is not always simplicial, and thus the $\dissection_\circ$-zonotope~$\Zono(\dissection_\circ)$ is not always simple. See \fref{fig:exmAccordiohedra}.

\begin{example}
Consider an accordion dissection~$\accordion_\circ$ (where each cell has at most $2$ edges which are internal diagonals of~$\polygon_\circ$). Label its internal diagonals by~$\delta_\circ^1, \dots, \delta_\circ^{|\accordion_\circ|}$ such that~$\delta_\circ^k$ and~$\delta_\circ^{k+1}$ belong to the same cell of~$\accordion_\circ$ for all~$k$. Identifying~$\b{e}_{\delta_\circ^k}$ to the simple root~$\b{f}_k - \b{f}_{k+1}$ of type~$A_{|\accordion_\circ|}$, the $\b{c}$-vectors of~$\allcvectors{\accordion_\circ}$ are all roots~$\pm (\b{f}_i - \b{f}_j) = \pm \sum_{i \le k \le j} \b{e}_{\delta_\circ^k}$ of type~$A_{|\accordion_\circ|}$. Therefore, the $\b{c}$-vector fan is the type~$A_{|\accordion_\circ|}$ Coxeter fan and the $\accordion_\circ$-zonotope is a permutahedron. More precisely,
\begin{align*}
\Zono(\accordion_\circ) & = \sum_{k \in [|\accordion_\circ|+1]} k(|\accordion_\circ|+1-k) \, [-\b{e}_{\delta_\circ^k}, \b{e}_{\delta_\circ^k}] = 2 \, \Perm(|\accordion_\circ|) - (|\accordion_\circ|+2) \sum_{i \in [|\accordion_\circ|+1]} \b{f}_i,
\end{align*}
where~$\Perm(|\accordion_\circ|) \eqdef \conv \bigset{ \sum_{i \in [|\accordion_\circ|+1]} \sigma(i) \, \b{f}_i}{\sigma \in \fS_{|\accordion_\circ|+1}}$ is the classical permutahedron.
\end{example}

The vertices of~$\Zono(\dissection_\circ)$ correspond to \defn{separable} subsets of~$\allcvectors{\dissection_\circ}$, \ie those which can be strictly separated from their complement by a hyperplane. Although we could work out all facets of~$\Zono(\dissection_\circ)$, we will only need the following specific inequalities.

\begin{proposition}
\label{prop:inequalitiesZonotope}
For any $\dissection_\circ$-accordion diagonal~$\gamma_\bullet$, the $\dissection_\circ$-zonotope~$\Zono(\dissection_\circ)$ has a facet defined by the inequality
\[
\dotprod{\biggvector{\dissection_\circ}{\gamma_\bullet}}{\b{x}} \le \bigrhs{\dissection_\circ}{\gamma_\bullet},
\]
where~$\rhs{\dissection_\circ}{\gamma_\bullet}$ is the \defn{$\dissection_\circ$-height} of~$\gamma_\bullet$, \ie the number of $\dissection_\circ$-accordion diagonals that cross~$\gamma_\bullet$.
\end{proposition}

\begin{proof}
Let~$\rhs{\dissection_\circ}{\gamma_\bullet}$ denote the maximum of~$\dotprod{\gvector{\dissection_\circ}{\gamma_\bullet}}{\b{x}}$ over~$\Zono(\dissection_\circ)$. As $\Zono(\dissection_\circ)$ is the Minkowski sum of all $\b{c}$-vectors, we have
\[
\bigrhs{\dissection_\circ}{\gamma_\bullet} = \sum_{\substack{\b{c} \in \allcvectors{\dissection_\circ} \\ \dotprod{\gvector{\dissection_\circ}{\gamma_\bullet}}{\b{c}} > 0}} \dotprod{\biggvector{\dissection_\circ}{\gamma_\bullet}}{\b{c}}.
\]
By Remark~\ref{rem:informalgcvectors}, we have~$\dotprod{\gvector{\dissection_\circ}{\gamma_\bullet}}{\b{c}} \in \{-1, 0, 1\}$ for any~$\b{c} \in \allcvectors{\dissection_\circ}$. We thus just need to count the distinct $\b{c}$-vectors~$\b{c}$ such that~$\dotprod{\gvector{\dissection_\circ}{\gamma_\bullet}}{\b{c}} > 0$. It turns out that it is more convenient and equivalent (since~$\allcvectors{\dissection_\circ} = -\allcvectors{\dissection_\circ}$) to count the distinct $\b{c}$-vectors~$\b{c}$ such that~$\dotprod{\gvector{\dissection_\circ}{\gamma_\bullet}}{\b{c}} < 0$. 
For that, let~$\zigzag_\circ$ denote the zigzag of the accordion crossed by~$\gamma_\bullet$ in~$\dissection_\circ$, and decompose~${\zigzag_\circ = \zigzag_\circ^- \sqcup \zigzag_\circ^+}$ such that~$\gvector{\dissection_\circ}{\gamma_\bullet} = \one_{\zigzag_\circ^+} - \one_{\zigzag_\circ^-}$ (where~$\one_{\rmX_\circ} \eqdef \sum_{\delta_\circ \in \rmX_\circ} \b{e}_{\delta_\circ}$ for~${\rmX_\circ \subseteq \dissection_\circ}$).

Let~$\delta_\bullet$ be a \mbox{$\dissection_\circ$-accordion} diagonal. Let~$\accordion_\circ^-$ (resp.~$\accordion_\circ^+$) denote the accordion crossed by~$\delta_\bullet = u_\bullet v_\bullet$ in~$\dissection_\circ$ and not including~$(u+1)_\circ$ or~$(v+1)_\circ$ (resp.~$(u-1)_\circ$ or~$(v-1)_\circ$). Let~$\b{c}^-(\delta_\bullet) \eqdef -\one_{\accordion_\circ^-}$ and~$\b{c}^+(\delta_\bullet) \eqdef \one_{\accordion_\circ^+}$. Recall from Lemma~\ref{lem:bijectioncvectors} that the negative (resp.~positive) $\b{c}$-vectors of~$\allcvectors{\dissection_\circ}$ are given by~$\b{c}^-(\delta_\bullet)$ (resp.~$\b{c}^+(\delta_\bullet)$) for all $\dissection_\circ$-accordion diagonal~$\delta_\bullet$ not in~$\dissection_\bullet^\mi$  (resp.~$\dissection_\bullet^\ma$).
We let the reader check that:
\begin{itemize}
\item If~$\gamma_\bullet$ and~$\delta_\bullet$ do not cross and have no common endpoint, both~$|\zigzag_\circ \cap \accordion_\circ^-|$ and~$|\zigzag_\circ \cap \accordion_\circ^+|$ are even. Thus $\dotprod{\gvector{\dissection_\circ}{\gamma_\bullet}}{\b{c}^-(\delta_\bullet)} = \dotprod{\gvector{\dissection_\circ}{\gamma_\bullet}}{\b{c}^+(\delta_\bullet)} = 0$.
\item If~$\gamma_\bullet$ and~$\delta_\bullet$ have a common endpoint, and~$\gamma_\bullet \delta_\bullet$ form a counterclockwise angle, then~$|{\zigzag_\circ \cap \accordion_\circ^-}|$ is even while $\zigzag_\circ \cap \accordion_\circ^+$ is empty or starts and ends in~$\zigzag_\circ^+$. Thus $\dotprod{\gvector{\dissection_\circ}{\gamma_\bullet}}{\b{c}^-(\delta_\bullet)} = 0$ while $\dotprod{\gvector{\dissection_\circ}{\gamma_\bullet}}{\b{c}^+(\delta_\bullet)} \ge 0$. The situation is similar if~$\gamma_\bullet \delta_\bullet$ form a clockwise angle.
\item If~$\gamma_\bullet$ and~$\delta_\bullet$ cross,~$\zigzag_\circ \cap \accordion_\circ^-$ and~${\zigzag_\circ \cap \accordion_\circ^+}$ are empty or start and end both in~$\zigzag_\circ^-$ or both in~$\zigzag_\circ^+$. Thus, either~$\dotprod{\gvector{\dissection_\circ}{\gamma_\bullet}}{\b{c}^-(\delta_\bullet)} < 0$ and~$\dotprod{\gvector{\dissection_\circ}{\gamma_\bullet}}{\b{c}^+(\delta_\bullet)} \ge 0$ or conversely.
\end{itemize}
We conclude from this case analysis that
\[
\rhs{\dissection_\circ}{\gamma_\bullet} = |\set{\b{c} \in \allcvectors{\dissection_\circ}}{\dotprod{\gvector{\dissection_\circ}{\gamma_\bullet}}{\b{c}} < 0}| = |\{\text{$\dissection_\circ$-accordion diagonals crossing~$\gamma_\bullet$}\}|.
\]
Finally, the inequality~$\dotprod{\gvector{\dissection_\circ}{\gamma_\bullet}}{\b{x}} \le \rhs{\dissection_\circ}{\gamma_\bullet}$ defines a priori a face~$\face(\gamma_\bullet)$ of the zonotope~$\Zono(\dissection_\circ)$. This face~$\face(\gamma_\bullet)$ is the Minkowski sum of the $\b{c}$-vectors of~$\allcvectors{\dissection_\circ}$ orthogonal to~$\gvector{\dissection_\circ}{\gamma_\bullet}$. Proposition~\ref{prop:gvectorscvectorsDualBases} ensures that any $\dissection_\circ$-accordion dissection~$\dissection_\bullet$ containing~$\gamma_\bullet$ already provides~$|\dissection_\bullet|-1$ linearly independent such $\b{c}$-vectors~$\cvector{\dissection_\circ}{\dissection_\bullet}{\delta_\bullet}$ for~${\delta_\bullet \in \dissection_\bullet \ssm \{\gamma_\bullet\}}$. We obtain that~$\face(\gamma_\bullet)$ has dimension~$|\dissection_\bullet|-1 = |\dissection_\circ|-1$ and is therefore a facet of the zonotope~$\Zono(\dissection_\circ)$.
\end{proof}

Define the half-space and the hyperplane corresponding to a solid $\dissection_\circ$-accordion diagonal~$\gamma_\bullet$~by
\begin{align*}
\bigHS{\dissection_\circ}{\gamma_\bullet} & \eqdef \bigset{\b{x} \in \R^{\dissection_\circ}}{\dotprod{\biggvector{\dissection_\circ}{\gamma_\bullet}}{\b{x}} \le \bigrhs{\dissection_\circ}{\gamma_\bullet}}, \\
\text{and}\qquad
\bigHyp{\dissection_\circ}{\gamma_\bullet} & \eqdef \bigset{\b{x} \in \R^{\dissection_\circ}}{\dotprod{\biggvector{\dissection_\circ}{\gamma_\bullet}}{\b{x}} = \bigrhs{\dissection_\circ}{\gamma_\bullet}}.
\end{align*}


\subsection{$\b{g}$-vector fan and $\dissection_\circ$-accordiohedron}
\label{subsec:gvectorFan}

In this section, we give a geometric realization of the $\dissection_\circ$-accordion complex. We start by realizing this simplicial complex as a complete simplicial fan in~$\R^{\dissection_\circ}$. We denote by~$\R_{\ge0} \b{R}$ the nonnegative span of a set~$\b{R}$ of vectors in~$\R^{\dissection_\circ}$.

\begin{theorem}
\label{thm:gvectorFan}
The collection of cones
\[
\gvectorFan(\dissection_\circ) \eqdef \bigset{\R_{\ge0} \biggvectors{\dissection_\circ}{\dissection_\bullet}}{\dissection_\bullet \text{ any $\dissection_\circ$-accordion dissection}}
\]
forms a complete simplicial fan, that we call the \defn{$\b{g}$-vector fan} of~$\dissection_\circ$.
\end{theorem}

The proof uses the following characterization of complete simplicial fans~\cite[Coro.~4.5.20]{DeLoeraRambauSantos}. We will provide as well an alternative proof in Remark~\ref{rem:simplerProofs} based on sections of Cambrian fans.

\begin{proposition}
\label{prop:characterizationFan}
Consider a pseudomanifold~$\Delta$ on a finite vertex set~$X$ and a set of vectors $\rays \eqdef (\ray_x)_{x \in X}$ of~$\R^d$. For~$\dissection \in \Delta$, define the cone~$\rays_\dissection \eqdef \set{\ray_x}{x \in \dissection}$. The collection of cones~$\bigset{\R_{\ge 0}\rays_\dissection}{\dissection \in \Delta}$ forms a complete simplicial fan if and~only~if
\begin{enumerate}
\item there exists a facet~$\dissection$ of~$\Delta$ such that~$\rays_\dissection$ is a basis of~$\R^d$ and such that the open cones~$\R_{> 0}\rays_\dissection$ and~$\R_{> 0}\b{R}_{\dissection'}$ are disjoint for any facet~$\dissection'$ of~$\Delta$ distinct from~$\dissection$;
\item for two adjacent facets~$\dissection, \dissection'$ of~$\Delta$ with~$\dissection \ssm \{x\} = \dissection' \ssm \{x'\}$, there is a linear dependence
\[
\alpha \, \ray_x + \alpha' \, \ray_{x'} + \sum_{y \in \dissection \cap \dissection'} \beta_y \, \ray_y = 0
\]
on~$\b{R}_{\dissection \cup \dissection'}$ where the coefficients~$\alpha$ and~$\alpha'$ have the same sign. (When these conditions hold, these coefficients do not vanish and the linear dependence is unique up to rescaling.)
\end{enumerate}
\end{proposition}

\begin{proof}[Proof of Theorem~\ref{thm:gvectorFan}]
By Corollary~\ref{coro:signCoherence}, the cone $\R_{\ge0} \gvectors{\dissection_\circ}{\dissection_\bullet^\mi}$ is the only cone of~$\gvectorFan(\dissection_\circ)$ intersecting the interior of the positive orthant~$(\R_{\ge0})^{\dissection_\circ}$. 
Consider now two adjacent maximal $\dissection_\circ$-accordion dissections~$\dissection_\bullet, \dissection_\bullet'$. Let~$\delta_\bullet \in \dissection_\bullet$ and~${\delta_\bullet' \in \dissection_\bullet'}$ be such that~$\dissection_\bullet \ssm \{\delta_\bullet\} = \dissection_\bullet' \ssm \{\delta_\bullet'\}$, and let~$\mu_\bullet$ and~$\nu_\bullet$ be the other diagonals as in Lemma~\ref{lem:flip} (see also \fref{fig:exmFlip}). Note that a diagonal of~$\dissection_\circ$ crosses none of (resp.~one of, resp.~both) the diagonals~$\delta_\bullet, \delta_\bullet'$ if and only if it crosses none of (resp.~one of, resp.~both) the diagonals~$\mu_\bullet, \nu_\bullet$. The same holds for a~$\ZZZ$ or a~$\SSS$ of~$\dissection_\circ$. Therefore, we have the linear dependence~$\gvector{\dissection_\circ}{\delta_\bullet} + \gvector{\dissection_\circ}{\delta_\bullet'} = \gvector{\dissection_\circ}{\mu_\bullet} + \gvector{\dissection_\circ}{\mu_\bullet}$. This shows that~$\gvectorFan(\dissection_\circ)$ satisfies the two conditions of Proposition~\ref{prop:characterizationFan}, and thus concludes the proof.
\end{proof}

\begin{remark}
The linear dependence $\gvector{\dissection_\circ}{\delta_\bullet} + \gvector{\dissection_\circ}{\delta_\bullet'} = \gvector{\dissection_\circ}{\mu_\bullet} + \gvector{\dissection_\circ}{\mu_\bullet}$ relating the $\b{g}$-vectors of two adjacent maximal \mbox{$\dissection_\circ$-accordion} dissections~$\dissection_\bullet, \dissection_\bullet'$ with~${\dissection_\bullet \ssm \{\delta_\bullet\} = \dissection_\bullet' \ssm \{\delta_\bullet'\}}$ shows that~${\det \big( \gvector{\dissection_\circ}{\dissection_\bullet} \big) = - \det \big( \gvector{\dissection_\circ}{\dissection_\bullet'} \big)}$. Since the initial cone~$\R_{\ge0} \gvectors{\dissection_\circ}{\dissection_\bullet^\mi}$ is generated by the coordinate vectors (see Example~\ref{exm:gcvectors}), we obtain that~$\det \big( \gvector{\dissection_\circ}{\dissection_\bullet} \big) = \pm 1$ for all \mbox{$\dissection_\circ$-accordion} dissection~$\dissection_\bullet$, so that the $\b{g}$-vector fan~$\gvectorFan(\dissection_\circ)$ is always \defn{smooth}. 
\end{remark}

By Proposition~\ref{prop:gvectorscvectorsDualBases}, any non-maximal cone of~$\gvectorFan(\dissection_\circ)$ is supported by a hyperplane orthogonal to a $\b{c}$-vector of~$\allcvectors{\dissection_\circ}$. We thus obtain the following consequence.

\begin{corollary}
The $\b{g}$-vector fan~$\gvectorFan(\dissection_\circ)$ coarsens the $\b{c}$-vector~fan~$\cvectorFan(\dissection_\circ)$.
\end{corollary}

\begin{example}
Following Example~\ref{exm:specialReferenceDissections}, we observe that special reference dissections give rise to the following relevant fans:
\begin{itemize}
\item For an accordion triangulation~$\accordion_\circ$ (\ie with no interior triangle), the $\b{g}$-vector fan~$\gvectorFan(\accordion_\circ)$ coincides with a type~$A$ Cambrian fan of N.~Reading and D.~Speyer~\cite{ReadingSpeyer}.
\item For an arbitrary triangulation~$\triangulation_\circ$ (with or without interior triangle), the $\b{g}$-vector fan~$\gvectorFan(\triangulation_\circ)$ was recently constructed in~\cite{HohlwegPilaudStella}.
\end{itemize}
\end{example}

\begin{example}
\fref{fig:exmgvectorFans} illustrates the $\b{g}$-vector fans~$\gvectorFan(\dissection_\circ)$ for various reference dissections~$\dissection_\circ$: the fan, the snake, and the cyclic triangulation of the hexagon, and a dissection of the heptagon. More precisely, we have represented the stereographic projection of the fans from the point~$[\, 1, 1, 1 \,]$. Therefore, the external face of the projection corresponds to the $\dissection_\circ$-accordion dissection~$\dissection_\bullet^\mi$. We have labeled all vertices of the projection (\ie the rays of the fan) by the corresponding $\dissection_\circ$-accordion diagonals.

\begin{figure}[t]
	\capstart
	\centerline{
		\begin{tabular}{l@{\hspace{.5cm}}l}
			\includegraphics[scale=1.3]{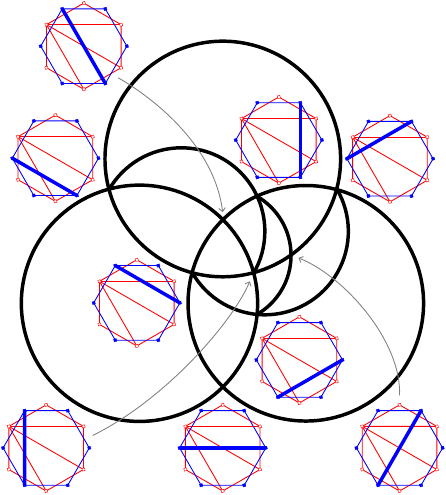} & \includegraphics[scale=1.3]{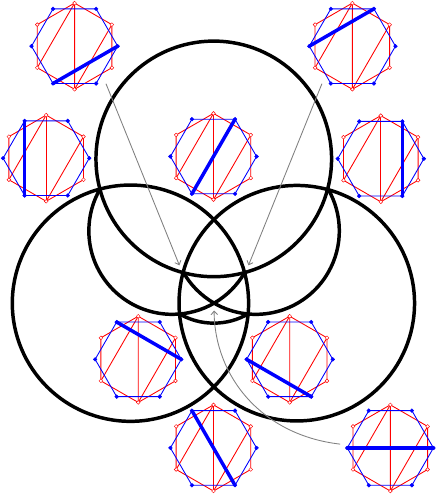} \\[.1cm]
			\includegraphics[scale=1.3]{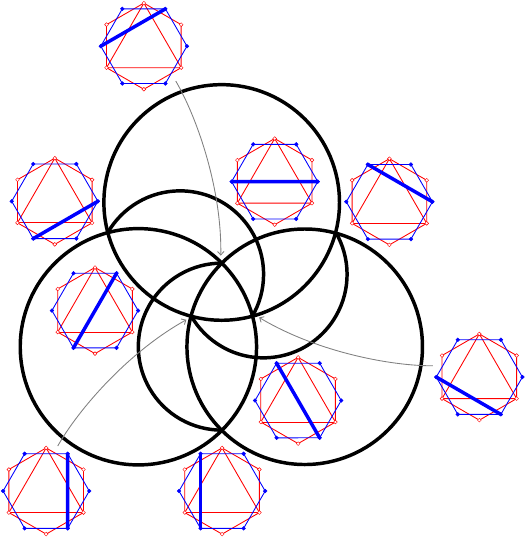} & \includegraphics[scale=1.3]{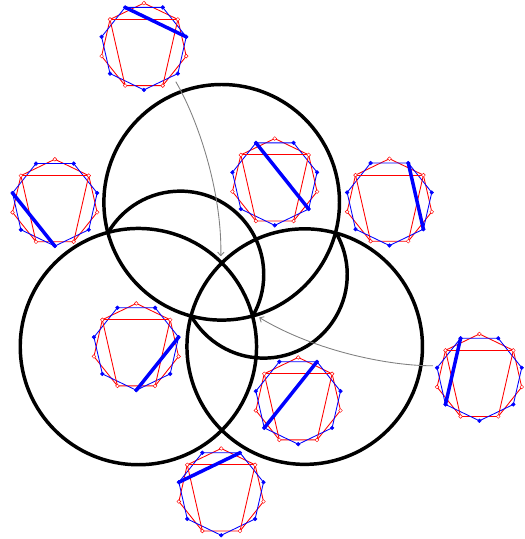}
		\end{tabular}
	}
	\caption{Stereographic projections of the $\b{g}$-vector fans~$\gvectorFan(\dissection_\circ)$ for various reference hollow dissections~$\dissection_\circ$. See \fref{fig:exmdvectorFans} for alternative simplicial fan realizations of these accordion complexes.}
	\label{fig:exmgvectorFans}
\end{figure}
\end{example}

We now provide a first polytopal realization of the $\b{g}$-vector fan~$\gvectorFan(\dissection_\circ)$ (see also Section~\ref{sec:projection}). This fan has a maximal cone for each maximal $\dissection_\circ$-accordion dissection and a ray for each $\dissection_\circ$-accordion diagonal. For a maximal $\dissection_\circ$-accordion dissection~$\dissection_\bullet$, we define a point~${\bigpoint{\dissection_\circ}{\dissection_\bullet} \in \R^{\dissection_\circ}}$~by
\[
\bigpoint{\dissection_\circ}{\dissection_\bullet} \eqdef \sum_{\delta_\bullet \in \dissection_\bullet} \bigrhs{\dissection_\circ}{\delta_\bullet} \cdot \bigcvector{\dissection_\circ}{\dissection_\bullet}{\delta_\bullet},
\]
where~$\rhs{\dissection_\circ}{\delta_\bullet}$ still denotes the \defn{$\dissection_\circ$-height} of~$\delta_\bullet$ defined as the number of $\dissection_\circ$-accordion diagonals that cross~$\delta_\bullet$. We will need the following two technical lemmas in the proof of Theorem~\ref{thm:accordiohedron}.

\begin{lemma}
\label{lem:pointInHyp}
For any maximal $\dissection_\circ$-accordion dissection~$\dissection_\bullet$, the point~$\point{\dissection_\circ}{\dissection_\bullet}$ is the intersection of all hyperplanes~$\Hyp{\dissection_\circ}{\delta_\bullet}$ with~$\delta_\bullet \in \dissection_\bullet$.
\end{lemma}

\begin{proof}
Observe first that the hyperplanes~$\Hyp{\dissection_\circ}{\delta_\bullet}$ with~$\delta_\bullet \in \dissection_\bullet$ have a unique intersection point, since~$\gvectors{\dissection_\circ}{\dissection_\bullet}$ is a basis. Moreover, since~$\gvectors{\dissection_\circ}{\dissection_\bullet}$ and~$\cvectors{\dissection_\circ}{\dissection_\bullet}$ form dual bases by Proposition~\ref{prop:gvectorscvectorsDualBases}, we have for any~$\gamma_\bullet \in \dissection_\bullet$:
\begin{align*}
\dotprod{\biggvector{\dissection_\circ}{\gamma_\bullet}}{\bigpoint{\dissection_\circ}{\dissection_\bullet}}
& = \sum_{\delta_\bullet \in \dissection_\bullet} \bigrhs{\dissection_\circ}{\delta_\bullet} \cdot \dotprod{\biggvector{\dissection_\circ}{\gamma_\bullet}}{\bigcvector{\dissection_\circ}{\dissection_\bullet}{\delta_\bullet}} \\
& = \sum_{\delta_\bullet \in \dissection_\bullet} \bigrhs{\dissection_\circ}{\delta_\bullet} \cdot \one_{\gamma_\bullet = \delta_\bullet} \; = \; \bigrhs{\dissection_\circ}{\gamma_\bullet}.
\qedhere
\end{align*}
\end{proof}

\begin{lemma}
\label{lem:factorcvectorFlip}
If~$\dissection_\bullet, \dissection_\bullet'$ are two adjacent maximal $\dissection_\circ$-accordion dissections, and~$\delta_\bullet \in \dissection_\bullet$ and~${\delta_\bullet' \in \dissection_\bullet'}$ are such that~${\dissection_\bullet \ssm \{\delta_\bullet\} = \dissection_\bullet' \ssm \{\delta_\bullet'\}}$, then
\[
\bigcvector{\dissection_\circ}{\dissection_\bullet}{\delta_\bullet} = - \bigcvector{\dissection_\circ}{\dissection_\bullet'}{\delta_\bullet'}
\quad\text{and}\quad
\bigpoint{\dissection_\circ}{\dissection_\bullet'} - \bigpoint{\dissection_\circ}{\dissection_\bullet} \in \Z_{<0} \cdot \bigcvector{\dissection_\circ}{\dissection_\bullet}{\delta_\bullet}.
\]
\end{lemma}

\begin{proof}
Let~$\dissection_\bullet, \dissection_\bullet'$ be two adjacent maximal $\dissection_\circ$-accordion dissections, let~$\delta_\bullet \in \dissection_\bullet$ and~${\delta_\bullet' \in \dissection_\bullet'}$ be such that~$\dissection_\bullet \ssm \{\delta_\bullet\} = \dissection_\bullet' \ssm \{\delta_\bullet'\}$, and let~$\mu_\bullet$ and~$\nu_\bullet$ be the other diagonals as in Lemma~\ref{lem:flip} (see also \fref{fig:exmFlip}). A quick case analysis then shows that
\[
\bigcvector{\dissection_\circ}{\dissection_\bullet'}{\gamma_\bullet} = 
\begin{cases}
\bigcvector{\dissection_\circ}{\dissection_\bullet}{\gamma_\bullet} & \text{ for all diagonal } \gamma_\bullet \in \dissection_\bullet \ssm \{\delta_\bullet, \mu_\bullet, \nu_\bullet\}, \\
-\bigcvector{\dissection_\circ}{\dissection_\bullet}{\delta_\bullet} & \text{ if } \gamma_\bullet = \delta_\bullet', \\
\bigcvector{\dissection_\circ}{\dissection_\bullet}{\gamma_\bullet} + \bigcvector{\dissection_\circ}{\dissection_\bullet}{\delta_\bullet} & \text{ if } \gamma_\bullet \in \{\mu_\bullet, \nu_\bullet\}.
\end{cases}
\]
Summing the contribution of all $\b{c}$-vectors with their coefficients~$\rhs{\dissection_\circ}{\gamma_\bullet}$, we obtain
\[
\bigpoint{\dissection_\circ}{\dissection_\bullet'} - \bigpoint{\dissection_\circ}{\dissection_\bullet} = \big( \bigrhs{\dissection_\circ}{\mu_\bullet} + \bigrhs{\dissection_\circ}{\nu_\bullet} - \bigrhs{\dissection_\circ}{\delta_\bullet} - \bigrhs{\dissection_\circ}{\delta_\bullet'} \big) \cdot \bigcvector{\dissection_\circ}{\dissection_\bullet}{\delta_\bullet}.
\]
Finally, note that any diagonal of~$\polygon_\bullet$ that crosses one of (resp.~both) the diagonals~$\mu_\bullet, \nu_\bullet$ also crosses one of (resp.~both) the diagonals~$\delta_\bullet, \delta_\bullet'$. Moreover, $\delta_\bullet$ and~$\delta_\bullet'$ cross each other but do not cross~$\mu_\bullet$ and~$\nu_\bullet$. It follows that~$\rhs{\dissection_\circ}{\mu_\bullet} + \rhs{\dissection_\circ}{\nu_\bullet} - \rhs{\dissection_\circ}{\delta_\bullet} - \rhs{\dissection_\circ}{\delta_\bullet'} \le -2 < 0$.
\end{proof}

\begin{theorem}
\label{thm:accordiohedron}
The $\b{g}$-vector fan is the normal fan of the \defn{$\dissection_\circ$-accordiohedron}~$\Acco(\dissection_\circ)$ defined equivalently as
\begin{itemize}
\item the convex hull of the points~$\point{\dissection_\circ}{\dissection_\bullet}$ for all maximal $\dissection_\circ$-accordion dissection~$\dissection_\bullet$, or
\item the intersection of the half-spaces~$\HS{\dissection_\circ}{\gamma_\bullet}$ for all $\dissection_\circ$-accordion diagonals~$\gamma_\bullet$.
\end{itemize}
Thus, the polar dual of~$\Acco(\dissection_\circ)$ is a polytopal realization of the $\dissection_\circ$-accordion complex~$\accordionComplex(\dissection_\circ)$.
\end{theorem}

The proof of Theorem~\ref{thm:accordiohedron} is based on the following characterization of polytopal realizations of a complete simplicial fan, whose proof can be found \eg in~\cite[Thm.~4.1]{HohlwegLangeThomas}.

\begin{theorem}
\label{theo:HohlwegLangeThomas}
Given a complete simplicial fan~$\Fan$ in~$\R^d$, consider for each ray~$\ray$ of~$\Fan$ a half-space~$\hs_\ray$ of~$\R^d$ containing the origin and defined by a hyperplane~$\hyp_\ray$ orthogonal to~$\ray$. For each maximal cone~$\Cone$ of~$\Fan$, let~$\b{a}(\Cone) \in \R^d$ be the intersection of all hyperplanes~$\hyp_\ray$ with~$\ray \in \Cone$. Then the following assertions are equivalent:
\begin{enumerate}[(i)]
\item The vector~$\b{a}(\Cone') - \b{a}(\Cone)$ points from~$\Cone$ to~$\Cone'$ for any two adjacent maximal cones~$\Cone$, $\Cone'$~of~$\Fan$.
\item The polytopes
\[
\conv\set{\b{a}(\Cone)}{\Cone \text{ maximal cone of } \Fan} \quad\text{ and }\quad
\bigcap_{\ray \text{ ray of } \Fan} \hs_\ray
\]
coincide and their normal fan is~$\Fan$.
\end{enumerate}
\end{theorem}

\begin{proof}[Proof of Theorem~\ref{thm:accordiohedron}]
The $\b{g}$-vector fan~$\gvectorFan(\dissection_\circ)$ has a ray~$\gvector{\dissection_\circ}{\delta_\bullet}$ for each $\dissection_\circ$-accordion diagonal~$\delta_\bullet$ and a maximal cone~$\Cone(\dissection_\bullet) = \R_{\ge0} \gvectors{\dissection_\circ}{\dissection_\bullet}$ for each maximal $\dissection_\circ$-accordion dissection~$\dissection_\bullet$. Consider the half-spaces~$\HS{\dissection_\circ}{\gamma_\bullet}$ for all $\dissection_\circ$-accordion diagonals~$\gamma_\bullet$. Lemma~\ref{lem:pointInHyp} ensures that the point~$\b{a}(\Cone(\dissection_\bullet))$ coincides with~$\point{\dissection_\circ}{\dissection_\bullet}$ for each maximal $\dissection_\circ$-accordion dissection~$\dissection_\bullet$. Finally, Lemma~\ref{lem:factorcvectorFlip} shows that the conditions of application of Theorem~\ref{theo:HohlwegLangeThomas} are fulfilled.
\end{proof}

\begin{example}
Following Example~\ref{exm:specialReferenceDissections}, observe that special reference hollow dissections give rise to the following relevant polytopes, illustrated in \fref{fig:exmAccordiohedra}:
\begin{itemize}
\item For a fan triangulation~$\triangulation_\circ$, the $\triangulation_\circ$-accordiohedron~$\Acco(\triangulation_\circ)$ is the classical associahedron constructed by S.~Shnider and S.~Sternberg~\cite{ShniderSternberg} and J.-L.~Loday~\cite{Loday}.
\item The $\accordion_\circ$-accordiohedra~$\Acco(\accordion_\circ)$ for all accordion triangulations~$\accordion_\circ$ are precisely the associahedra constructed by C.~Hohlweg and C.~Lange in~\cite{HohlwegLange}.
\item For a triangulation~$\triangulation_\circ$ with an interior triangle, the $\triangulation_\circ$-accordiohedron~$\Acco(\triangulation_\circ)$ was recently constructed in~\cite{HohlwegPilaudStella}. For example, for the triangulation of the hexagon with an interior triangle, this associahedron appeared as a mysterious realization in~\cite{CeballosSantosZiegler}.
\item For a quadrangulation~$\quadrangulation_\circ$, the $\quadrangulation_\circ$-accordiohedron~$\Acco(\quadrangulation_\circ)$ is a realization of the Stokes polytope announced by Y.~Baryshnikov~\cite{Baryshnikov} and discussed by F.~Chapoton in~\cite{Chapoton-quadrangulations}.
\end{itemize}

\begin{figure}[t]
	\capstart
	\centerline{\includegraphics[width=\textwidth]{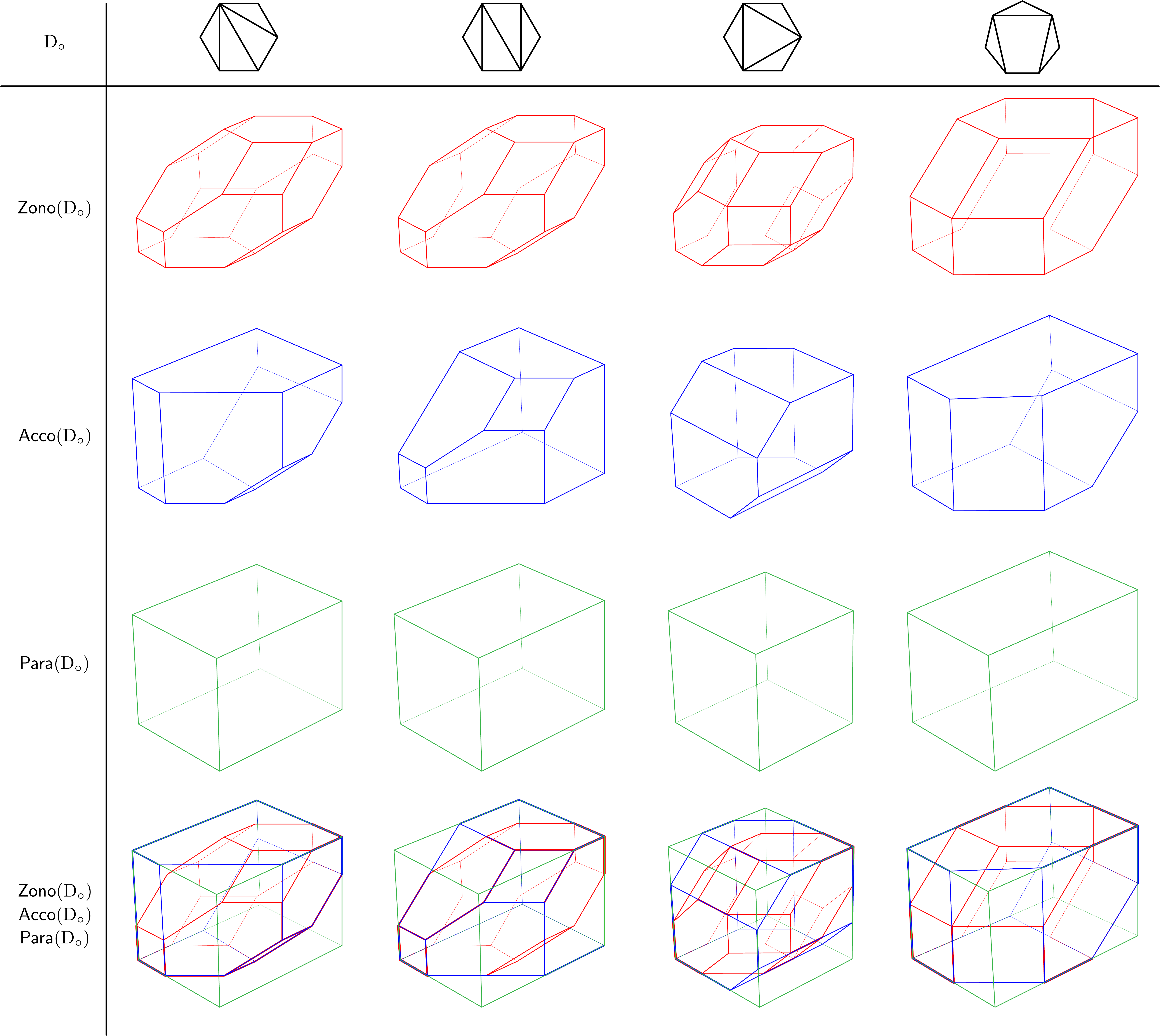}}
	\caption{The zonotope~$\Zono(\dissection_\circ)$, $\dissection_\circ$-accordiohedron~$\Acco(\dissection_\circ)$ and parallelepiped~$\Para(\dissection_\circ)$ for different reference dissections~$\dissection_\circ$. The first column is J.-L.~Loday's associahedron~\cite{Loday}, the second column is one of C.~Hohlweg and C.~Lange's associahedra~\cite{HohlwegLange}, the third column appeared in a discussion in C.~Ceballos, F.~Santos and G.~Ziegler's survey on associahedra~\cite[Fig.~3]{CeballosSantosZiegler} and was explained in C.~Hohlweg, V.~Pilaud and S.~Stella's recent paper~\cite{HohlwegPilaudStella}, and the last column is a Stokes complex discussed by F.~Chapoton in~\cite{Chapoton-quadrangulations} and illustrated in \fref{fig:exmAccordionComplex}.}
	\label{fig:exmAccordiohedra}
\end{figure}
\end{example}

We conclude this section by an immediate consequence of Theorem~\ref{thm:accordiohedron}. To our knowledge, this property of accordion complexes was not observed before. However, using the connection between accordion complexes and support $\tau$-tilting complexes~\cite{GarverMcConville, PaluPilaudPlamondon, PilaudPlamondonStella, BrustleDouvilleMousavandThomasYildirim}, it can also be obtained from~\cite[Thm.~1.7]{DemonetIyamaJasso}.

\begin{corollary}
For any reference dissection~$\dissection_\circ$, the $\dissection_\circ$-accordion complex~$\accordionComplex(\dissection_\circ)$ is shellable.
\end{corollary}


\subsection{Some properties of~$\Acco(\dissection_\circ)$}

We conclude this section by pointing out some relevant combinatorial and geometric properties and observations on the $\dissection_\circ$-accordiohedron.

\begin{proposition}
The graph of the $\dissection_\circ$-accordiohedron~$\Acco(\dissection_\circ)$ linearly oriented in the direction~$-\one \eqdef - \sum_{\delta_\circ \in \dissection_\circ} \b{e}_{\delta_\circ}$ is the Hasse diagram of the accordion lattice~$\accordionLattice(\dissection_\circ)$.
\end{proposition}

\begin{proof}
Consider two adjacent maximal $\dissection_\circ$-accordion dissections~$\dissection_\bullet, \dissection_\bullet'$ such that the flip from~$\dissection_\bullet$ to~$\dissection_\bullet'$ is increasing. Let~$\delta_\bullet \in \dissection_\bullet$ and~${\delta_\bullet' \in \dissection_\bullet'}$ be such that~${\dissection_\bullet \ssm \{\delta_\bullet\} = \dissection_\bullet' \ssm \{\delta_\bullet'\}}$. As observed in Remark~\ref{rem:informalgcvectors}\,(ii), the $\b{c}$-vector~$\cvector{\dissection_\circ}{\dissection_\bullet}{\delta_\bullet}$ is the characteristic vector~$\one_{\accordion_\circ}$ of the set~$\accordion_\circ$ of diagonals of~$\dissection_\circ$ crossed by both~$\delta_\bullet$ and~$\delta_\bullet'$. Applying Lemma~\ref{lem:factorcvectorFlip}, we therefore obtain that
\[
\dotprod{-\one}{\bigpoint{\dissection_\circ}{\dissection_\bullet'} - \bigpoint{\dissection_\circ}{\dissection_\bullet}} = \dotprod{-\one}{\lambda \cdot \bigcvector{\dissection_\circ}{\dissection_\bullet}{\delta_\bullet}} = \lambda \cdot \dotprod{-\one}{\one_{\accordion_\circ}} = -\lambda \cdot |\accordion_\circ|,
\]
for some~$\lambda \in \Z_{<0}$. Thus, the linear functional~$-\one$ indeed orients the edge~$[\point{\dissection_\circ}{\dissection_\bullet}, \point{\dissection_\circ}{\dissection_\bullet'}]$ from~$\point{\dissection_\circ}{\dissection_\bullet}$ to~$\point{\dissection_\circ}{\dissection_\bullet'}$.
\end{proof}

\begin{remark}
Since the $\b{c}$-vector fan~$\cvectorFan(\dissection_\circ)$ refines the $\b{g}$-vector fan~$\gvectorFan(\dissection_\circ)$, there is a natural projection~$\projection$ from the vertices of the $\dissection_\circ$-zonotope~$\Zono(\dissection_\circ)$ to that of the $\dissection_\circ$-accordiohedron~$\Acco(\dissection_\circ)$. In analogy to the acyclic case, one could hope to obtain the accordion lattice as a lattice quotient through this projection. However, the transitive closure of the graph of the $\dissection_\circ$-zonotope~$\Zono(\dissection_\circ)$ oriented in the direction~$-\one$ is not a lattice in general (the first counter-example is the dissection with a central square surrounded by~$4$ triangles). As shown in~\cite{GarverMcConville}, the right objects are not the separable subsets of $\b{c}$-vectors (\ie the vertices of~$\Zono(\dissection_\circ)$) but the biclosed subsets of $\b{c}$-vectors.
\end{remark}

\begin{proposition}
\label{prop:parallelFacets}
The accordiohedron~$\Acco(\dissection_\circ)$ has precisely~$|\dissection_\circ|$ pairs of parallel facets.
\end{proposition}

\begin{proof}
Two facets of~$\Acco(\dissection_\circ)$ are parallel if and only if the corresponding $\b{g}$-vectors are opposite. We therefore want to prove that the pairs of opposite coordinate vectors are the only pairs of opposite $\b{g}$-vectors. Assume by contradiction that there exist two hollow diagonals~$\delta_\circ, \delta_\circ' \in \dissection_\circ$ and two solid $\dissection_\circ$-diagonals~$\delta_\bullet, \delta_\bullet'$ such that~$\gvector{\dissection_\circ}{\delta_\bullet}$ and~$\gvector{\dissection_\circ}{\delta_\bullet'}$ have non-zero opposite coordinate both on~$\delta_\circ$ and~$\delta_\circ'$. Then both~$\delta_\bullet$ and~$\delta_\bullet'$ cross both~$\delta_\circ$ and~$\delta_\circ'$. But this implies that they both slalom on~$\delta_\circ$ (and on~$\delta_\circ'$) in the same way. Contradiction.
\end{proof}

Recall from Example~\ref{exm:gcvectors} that the $\b{g}$-vectors of the diagonals of~$\dissection_\bullet^\mi$ (resp.~$\dissection_\bullet^\ma$) are the coordinate vectors (resp.~negative of the coordinate vectors). Consider the \defn{$\dissection_\circ$-parallelepiped}
\[
\Para(\dissection_\circ) \eqdef \set{\b{x} \in \R^{\dissection_\circ}}{\dotprod{\gvector{\dissection_\circ}{\delta_\bullet}}{\b{x}}| \le \rhs{\dissection_\circ}{\delta_\bullet}\text{ for all } \delta_\bullet \in \dissection_\bullet^\mi \cup \dissection_\bullet^\ma}
\]
defined by the inequalities of the $\dissection_\circ$-zonotope~$\Zono(\dissection_\circ)$ corresponding to the positive and negative basis vectors. Our next statement follows from Proposition~\ref{prop:parallelFacets} and is illustrated in \fref{fig:exmAccordiohedra}.

\begin{corollary}
For any~$\dissection_\circ$, we have matriochka polytopes: 
\(
\Zono(\dissection_\circ) \subseteq \Acco(\dissection_\circ) \subseteq \Para(\dissection_\circ).
\)
\end{corollary}

In fact, each polytope in this chain is obtained by deleting facets from the previous one.

Consider now an isometry~$\sigma$ of the plane that preserves the hollow polygon~$\polygon_\circ$ and the solid polygon~$\polygon_\bullet$. For any diagonals and dissections~$\delta_\bullet \in \dissection_\bullet$ and~$\delta_\circ \in \dissection_\circ$, we~have
\begin{itemize}
\item $\delta_\bullet$ is a $\dissection_\circ$-accordion diagonal $\iff$ $\sigma(\delta_\bullet)$ is a $\sigma(\dissection_\circ)$-accordion diagonal,
\item $\dissection_\bullet$ is a $\dissection_\circ$-accordion dissection $\iff$ $\sigma(\dissection_\bullet)$ is a $\sigma(\dissection_\circ)$-accordion dissection,
\item if~$\Sigma : \R^{\dissection_\circ} \to \R^{\sigma(\dissection_\circ)}$ denotes the isometry defined by
\(
\big( \Sigma(\b{x}) \big)_{\sigma(\delta_\circ)} \eqdef \signature(\sigma) \cdot \b{x}_{\delta_\circ},
\)
(where~${\signature(\sigma) = 1}$ if~$\sigma$ is direct and~$-1$ if~$\sigma$ is indirect), then we have
\[
\begin{array}{@{\hspace{1.5cm}}r@{\;=\;}lcr@{\;=\;}l}
\biggvector{\sigma(\dissection_\circ)}{\sigma(\delta_\bullet)} & \Sigma\big(\gvector{\dissection_\circ}{\delta_\bullet}\big),
& \qquad\qquad &
\bigcvector{\sigma(\dissection_\circ)}{\sigma(\dissection_\bullet)}{\sigma(\delta_\bullet)} & \Sigma\big(\cvector{\dissection_\circ}{\dissection_\bullet}{\delta_\bullet}\big), \\[.1cm]
\bigrhs{\sigma(\dissection_\circ)}{\sigma(\delta_\bullet)} & \bigrhs{\dissection_\circ}{\delta_\bullet},
& \text{and} &
\bigpoint{\sigma(\dissection_\circ)}{\sigma(\dissection_\bullet)} & \Sigma\big(\point{\dissection_\circ}{\dissection_\bullet}\big).
\end{array}
\]
\end{itemize}
This immediately implies the following statement.

\begin{proposition}
Any $\polygon_\circ$-preserving isometry~$\sigma : \R^2 \to \R^2$ induces an isometry~$\Sigma : \R^{\dissection_\circ} \to \R^{\sigma(\dissection_\circ)}$
with $\Sigma\big(\Zono(\dissection_\circ)\big) \! = \Zono\big(\sigma(\dissection_\circ)\big)$, $\Sigma\big(\Acco(\dissection_\circ)\big) \! = \Acco\big(\sigma(\dissection_\circ)\big)$ and~$\Sigma\big(\Para(\dissection_\circ)\big) \! = \Para\big(\sigma(\dissection_\circ)\big)$.
\end{proposition}


We say that a dissection~$\dissection$ is \defn{$\sigma$-invariant} when~$\sigma(\dissection) = \dissection$. Assume now that~$\sigma$ is a rotation and~$\dissection_\circ$ is $\sigma$-invariant. We call \defn{$\sigma$-invariant $\dissection_\circ$-accordion complex} the simplicial complex~$\accordionComplex^\sigma(\dissection_\circ)$ whose vertices are the crossing-free $\sigma$-orbits of $\dissection_\circ$-accordion diagonals, and whose faces are sets of such orbits whose union is crossing-free. In other words, the faces of~$\accordionComplex^\sigma(\dissection_\circ)$ are $\sigma$-invariant $\dissection_\circ$-accordion dissections, seen as sets of $\sigma$-orbits of diagonals. 

\begin{lemma}
The $\sigma$-invariant $\dissection_\circ$-accordion complex~$\accordionComplex^\sigma(\dissection_\circ)$ is a pseudomanifold.
\end{lemma}

\begin{proof}
Assume first that~$\sigma$ is the central symmetry. In this case, there are two possible types of orbits: the long $\dissection_\circ$-accordion diagonals and the centrally symmetric pairs of $\dissection_\circ$-accordion diagonals. One can check that any facet of~$\accordionComplex^\sigma(\dissection_\circ)$ has a long diagonal if and only if~$\dissection_\circ$ has, and has as many centrally symmetric pairs of diagonals as~$\dissection_\circ$. Finally, any orbit in any facet of~$\accordionComplex^\sigma(\dissection_\circ)$ can be flipped: long diagonals can already be flipped in~$\accordionComplex(\dissection_\circ)$, and a centrally symmetric pair of diagonals can be flipped by flipping one after the other its~two~diagonals in~$\accordionComplex(\dissection_\circ)$.

Finally, the general statement follows from this special case. Indeed, if~$\sigma$ is not a central symmetry, let~$\cell_\circ$ denote the cell of~$\dissection_\circ$ containing the center of~$\polygon_\circ$, let~$u_\circ$ be a vertex of~$\cell_\circ$, let~$\underline{\dissection}_\circ$ be the set of diagonals of~$\dissection_\circ$ whose endpoints are between~$u_\circ$ and~$\sigma(u_\circ)$, and let~$\rho$ be the central symmetry around the middle of~$u_\circ \sigma(u_\circ)$. Then~$\accordionComplex^\sigma(\dissection_\circ)$ is isomorphic to~${\accordionComplex^\rho \big( \underline{\dissection}_\circ \cup \rho(\underline{\dissection}_\circ) \big)}$.
\end{proof}

Let~$\Sigma : \R^{\dissection_\circ} \to \R^{\dissection_\circ}$ denote the isometry defined by~$\big( \Sigma(\b{x}) \big)_{\sigma(\delta_\circ)} \eqdef \b{x}_{\delta_\circ}$ and~$\fix{\Sigma}$ denote the linear subspace of fixed points of~$\Sigma$. According to the previous discussion, a maximal $\dissection_\circ$-accordion dissection~$\dissection_\bullet$ is $\sigma$-invariant if and only if~$\point{\dissection_\circ}{\dissection_\bullet} \in \fix{\Sigma}$. We obtain the following statement.

\begin{proposition}
For a $\sigma$-invariant dissection~$\dissection_\circ$, the polytope~$\Acco^\sigma(\dissection_\circ)$ defined equivalently~as
\begin{itemize}
\item the convex hull of~$\point{\dissection_\circ}{\dissection_\bullet}$ for all $\sigma$-invariant maximal $\dissection_\circ$-accordion dissections~$\dissection_\bullet$,
\item the intersection of the $\dissection_\circ$-accordiohedron~$\Acco(\dissection_\circ)$ with the fixed space~$\fix{\Sigma}$,
\end{itemize}
is a polytopal realization of the $\sigma$-invariant accordion complex~$\accordionComplex^\sigma(\dissection_\circ)$.
\end{proposition}

\begin{proof}
Denote by~$P = \conv \set{\point{\dissection_\circ}{\dissection_\bullet}}{\text{$\sigma$-invariant maximal $\dissection_\circ$-accordion dissections~$\dissection_\bullet$}}$ and by ${Q = \Acco(\dissection_\circ) \cap \fix{\Sigma}}$. The inclusion~$P \subseteq Q$ is clear since $\dissection_\bullet$ is $\sigma$-invariant if and only if~$\point{\dissection_\circ}{\dissection_\bullet} \in \fix{\Sigma}$. We now prove the reverse inclusion. For that, consider an arbitrary $\sigma$-invariant maximal $\dissection_\circ$-accordion dissection~$\dissection_\bullet$. Its corresponding point~$\point{\dissection_\circ}{\dissection_\bullet}$ is a common vertex of~$P$ and~$Q$. Moreover, any edge~$e$ of~$Q$ incident to~$\point{\dissection_\circ}{\dissection_\bullet}$ is the intersection of $\fix{\Sigma}$ with a face~$F$ of~$\Acco(\dissection_\circ)$ that corresponds to a $\sigma$-invariant $\dissection_\circ$-dissection. Since~$\accordionComplex^\sigma(\dissection_\circ)$ is a pseudomanifold, this dissection can be refined into another maximal $\sigma$-invariant $\dissection_\circ$-accordion dissection~$\dissection_\bullet'$. The point~$\point{\dissection_\circ}{\dissection_\bullet'}$ belongs to~$F$ and to~$\fix{\Sigma}$ and thus to~$e$. We conclude that if~$v$ is a common vertex of~$P$ and~$Q$, then so are all neighbors of~$v$ in the graph of~$Q$. Propagating this property, we obtain that all vertices of~$Q$ are also vertices of~$P$, so that~$P = Q$. Finally, there is a clear injection from the $\sigma$-invariant accordion complex~$\accordionComplex^\sigma(\dissection_\circ)$ to the boundary complex of~$P = Q$, thus a bijection (since these complexes are two spheres with the same vertex set).
\end{proof}


\section{The $\b{d}$-vector fan}
\label{sec:dvectorFan}

In this section, we discuss the generalization to the $\dissection_\circ$-accordion complex of another classical geometric realization of the associahedron coming from the theory of cluster algebras~\cite{FominZelevinsky-ClusterAlgebrasI, FominZelevinsky-ClusterAlgebrasII, ChapotonFominZelevinsky, CeballosSantosZiegler}. Namely, we define compatibility vectors in analogy with the denominator vectors of cluster variables, and we characterize the reference dissections~$\dissection_\circ$ for which these vectors support a complete simplicial fan realizing the $\dissection_\circ$-accordion complex.

\subsection{$\b{d}$-vectors}

Fix a dissection~$\dissection_\circ$ of the hollow $n$-gon. For a hollow diagonal~$\delta_\circ = i_\circ j_\circ$ and a solid diagonal~$\delta_\bullet$, we denote by
\[
\comp{\delta_\circ}{\delta_\bullet}
\eqdef
\begin{cases}
	-1 & \text{if } \delta_\bullet = (i-1)_\bullet (j-1)_\bullet, \\
	 0 & \text{if } \delta_\bullet \text{ and } (i-1)_\bullet (j-1)_\bullet \text{ do not cross,} \\
	 1 & \text{if } \delta_\bullet \text{ and } (i-1)_\bullet (j-1)_\bullet \text{ cross.}
\end{cases}
\]
For any $\dissection_\circ$-accordion diagonal~$\delta_\bullet$, the \defn{$\b{d}$-vector} of~$\delta_\bullet$ with respect to~$\dissection_\circ$ is the vector
\[
\bigdvector{\dissection_\circ}{\delta_\bullet} = \sum_{\delta_\circ \in \dissection_\circ} \comp{\delta_\circ}{\delta_\bullet} \, \b{e}_{\delta_\circ}.
\]
In other words, our~$\b{d}$-vector~$\b{d}(\dissection_\circ\,|\,\delta_\bullet)$ records the compatibility of the diagonal~$\delta_\bullet$ with the dissection~$\dissection_\bullet^\mi$. For a $\dissection_\circ$-accordion dissection~$\dissection_\bullet$, we define~$\bigdvectors{\dissection_\circ}{\dissection_\bullet} \eqdef \bigset{\bigdvector{\dissection_\circ}{\delta_\bullet}}{\delta_\bullet \in \dissection_\bullet}$.

\begin{example}
Consider the hollow dissection~$\dissection_\circ^\ex = \{3_\circ 7_\circ, 3_\circ 13_\circ, 9_\circ 13_\circ\}$ and the rightmost solid dissection~$\dissection_\bullet^\ex = \{2_\bullet 6_\bullet, 2_\bullet 10_\bullet, 10_\bullet 14_\bullet\}$ of \fref{fig:exmAccordionDissections}. Its $\b{d}$-vectors are given by
\[
\bigdvector{\dissection_\circ^\ex}{2_\bullet 6_\bullet} = - \b{e}_{3_\circ 7_\circ},
\quad
\bigdvector{\dissection_\circ^\ex}{2_\bullet 10_\bullet} = \b{e}_{9_\circ 13_\circ},
\quad\text{and}\quad
\bigdvector{\dissection_\circ^\ex}{10_\bullet 14_\bullet} = \b{e}_{3_\circ 13_\circ} + \b{e}_{9_\circ 13_\circ}.
\]
\end{example}

\subsection{$\b{d}$-vector fan}

We now consider the set of cones
\[
\bigset{\R_{\ge0} \bigdvectors{\dissection_\circ}{\dissection_\bullet}}{\dissection_\bullet \text{ any $\dissection_\circ$-accordion dissection}}
\]
generated by the $\b{d}$-vectors of the $\dissection_\circ$-accordion dissections. We want to characterize the reference hollow dissections~$\dissection_\circ$ for which these cones form a complete simplicial fan realizing the $\dissection_\circ$-accordion complex. We start with a negative result. An \defn{even interior cell} of a dissection~$\dissection$ is a cell with an even number of edges which are all internal diagonals of~$\dissection$.

\begin{proposition}
If the reference hollow dissection~$\dissection_\circ$ contains an even interior cell, then the $\b{d}$-vectors cannot realize the $\dissection_\circ$-accordion complex.
\end{proposition}

\begin{proof}
Assume that~$\dissection_\circ$ contains an even interior cell~$\cell_\circ$. Denote its vertices by~$i_\circ^1, \dots, i_\circ^{2p}$ (in clockwise order) and its edges~$\delta_\circ^k \eqdef i_\circ^k i_\circ^{k+1}$ for~$k \in [2p]$ (where~$i^{2p+1} = i^1$ by convention). Denote by~$\dissection_\circ^k$ the set of diagonals of~$\dissection_\circ$ separated form~$\cell_\circ$ by~$\delta_\circ^k$ (including~$\delta_\circ^k$ itself), and let~$\dissection_\bullet^k \eqdef \set{(i-1)_\bullet (j-1)_\bullet}{i_\circ j_\circ \in \dissection_\circ^k}$. Consider the solid diagonals~$\delta_\bullet^k \eqdef (i^k+1)_\bullet (i^{k+1}+1)_\bullet$ for~$k \in [2p]$. Observe that~$\delta_\bullet^k$ only crosses diagonals of~$\dissection_\bullet^{k-1}$ and $\dissection_\bullet^k$, and that~$\delta_\bullet^k$ and~$\delta_\bullet^{k+1}$ cross precisely the same diagonals of~$\dissection_\bullet^k$. Since the cell is even, it ensures that the $\b{d}$-vectors of the diagonals~$\delta_\bullet^k$ for~$k \in [2p]$ satisfy the linear dependence
\[
\sum_{\substack{k \in [2p] \\ k \text{ even}}} \bigdvector{\dissection_\circ}{\delta_\bullet^k} = \sum_{\substack{k \in [2p] \\ k \text{ odd}}} \bigdvector{\dissection_\circ}{\delta_\bullet^k}.
\]
However, as already mentioned in Section~\ref{subsec:accordionLattice}, the diagonals~$\delta_\bullet^k$ for~$k \in [2p]$ all belong to the $\dissection_\circ$-accordion dissection~$\dissection_\bullet^\ma \eqdef \set{(i+1)_\bullet (j+1)_\bullet}{i_\circ j_\circ \in \dissection_\circ}$. Therefore, the cone~$\R_{\ge0} \dvectors{\dissection_\circ}{\dissection_\bullet^\ma}$ is degenerate, so that the $\b{d}$-vectors cannot realize the $\dissection_\circ$-accordion complex.
\end{proof}

\begin{example}
Consider a hollow octagon and the reference dissection~$\dissection_\circ \eqdef \{ 1_\circ 5_\circ, 5_\circ 9_\circ, 9_\circ 13_\circ, 13_\circ 1_\circ \}$ with an interior square cell~$1_\circ 5_\circ 9_\circ 13_\circ$. Then we have
\begin{gather*}
\bigdvector{\dissection_\circ}{2_\bullet 6_\bullet} = \b{e}_{1_\circ 5_\circ} + \b{e}_{5_\circ 9_\circ} \qquad
\bigdvector{\dissection_\circ}{6_\bullet 10_\bullet} = \b{e}_{5_\circ 9_\circ} + \b{e}_{9_\circ 13_\circ} \\
\bigdvector{\dissection_\circ}{10_\bullet 14_\bullet} = \b{e}_{9_\circ 13_\circ} + \b{e}_{13_\circ 1_\circ} \qquad
\bigdvector{\dissection_\circ}{14_\bullet 2_\bullet} = \b{e}_{13_\circ 1_\circ} + \b{e}_{1_\circ 5_\circ}
\end{gather*}
so that there is already a linear dependence
\[
 \bigdvector{\dissection_\circ}{2_\bullet 6_\bullet}
+\bigdvector{\dissection_\circ}{10_\bullet 14_\bullet}
=
 \bigdvector{\dissection_\circ}{6_\bullet 10_\bullet}
+\bigdvector{\dissection_\circ}{14_\bullet 2_\bullet}
\]
among the $\b{d}$-vectors of the $\dissection_\circ$-accordion dissection~$\dissection_\bullet^\ma = \{2_\bullet 6_\bullet, 6_\bullet 10_\bullet, 10_\bullet 14_\bullet, 14_\bullet 2_\bullet\}$.
\end{example}

On the negative side, we have seen that the presence of even interior cells prohibits the $\b{d}$-vectors from forming a complete simplicial fan. The positive side is that the even interior cells are the only obstructions.

\begin{theorem}
\label{thm:dvectorFan}
The collection of cones
\[
\dvectorFan(\dissection_\circ) \eqdef \bigset{\R_{\ge0} \bigdvectors{\dissection_\circ}{\dissection_\bullet}}{\dissection_\bullet \text{ any $\dissection_\circ$-accordion dissection}}
\]
forms a complete simplicial fan, that we call the \defn{$\b{d}$-vector fan} of~$\dissection_\circ$, if and only if~$\dissection_\circ$ contains no even interior cell.
\end{theorem}

\begin{proof}
We use the characterization of complete simplicial fans presented in Proposition~\ref{prop:characterizationFan}.

Observe first that~$\dvectors{\dissection_\circ}{\dissection_\bullet^\mi} = (\R_{\le0})^{\dissection_\circ}$ is the only cone of~$\dvectorFan(\dissection_\circ)$ intersecting the interior of the negative orthant~$(\R_{\le0})^{\dissection_\circ}$. Therefore, $\dvectorFan(\dissection_\circ)$ fulfills Condition~(1) in Proposition~\ref{prop:characterizationFan}.

To check Condition~(2), consider two adjacent maximal $\dissection_\circ$-accordion dissections~$\dissection_\bullet$ and~$\dissection_\bullet'$ and let~${\delta_\bullet \in \dissection_\bullet}$ and~$\delta_\bullet' \in \dissection_\bullet'$ be such that~$\dissection_\bullet \ssm \{\delta_\bullet\} = \dissection_\bullet' \ssm \{\delta_\bullet'\}$. Let~$\mu_\bullet$ and~$\nu_\bullet$ be the diagonals of~$\overline{\dissection}_\bullet \cap \overline{\dissection}_\bullet'$ as in Lemma~\ref{lem:flip} (see also \fref{fig:exmFlip}). In other words, $\mu_\bullet$ and~$\nu_\bullet$ are incident to both~$\delta_\bullet$ and~$\delta_\bullet'$, and they are crossed by the hollow diagonal which intersect $\delta_\bullet$ and~$\delta_\bullet'$. Let~$\gamma_\circ = i_\circ j_\circ$ be such a hollow diagonals crossing~$\delta_\bullet, \delta_\bullet', \mu_\bullet$ and~$\nu_\bullet$, and let~$\gamma_\bullet = (i-1)_\bullet (j-1)_\bullet$. We now distinguish three cases:

\begin{itemize}
\item Assume that $\gamma_\bullet$ still crosses~$\mu_\bullet$ and~$\nu_\bullet$. In this case, any diagonal of~$\dissection_\bullet^\mi$ crossing both (resp.~either) $\delta_\bullet$ and (resp.~or) $\delta_\bullet'$ also crosses both (resp.~either) $\mu_\bullet$ and (resp.~or) $\nu_\bullet$. See \fref{fig:proofdvectorFan}\,(left). Therefore, the $\b{d}$-vectors of~$\dissection_\bullet \cup \dissection_\bullet'$ satisfy the linear dependence
\[
\dvector{\dissection_\circ}{\delta_\bullet} + \dvector{\dissection_\circ}{\delta_\bullet'} = \dvector{\dissection_\circ}{\mu_\bullet} + \dvector{\dissection_\circ}{\nu_\bullet}.
\]

\item Assume that~$\gamma_\bullet$ crosses neither~$\mu_\bullet$ nor~$\nu_\bullet$. Then~$\gamma_\bullet$ is incident to both~$\mu_\bullet$ and~$\nu_\bullet$, and therefore is either~$\delta_\bullet$ or~$\delta_\bullet'$, say~$\gamma_\bullet = \delta_\bullet$. Then~$\dvector{\gamma_\circ}{\delta_\bullet} = -1$ while~$\dvector{\gamma_\circ}{\delta_\bullet'} = 1$ (since~$\delta_\bullet'$ crosses~$\delta_\bullet = \gamma_\bullet$), so that~$\dvector{\gamma_\circ}{\delta_\bullet} + \dvector{\gamma_\circ}{\delta_\bullet'} = 0$. Moreover, we have~$\dvector{\gamma_\circ}{\delta_\bullet'} = 0$ for any diagonal~$\varepsilon_\bullet \in \dissection_\bullet \cap \dissection_\bullet'$ since~$\delta_\bullet = \gamma_\bullet$ cannot cross~$\varepsilon_\bullet$ as they both belongs to~$\dissection_\bullet$. Therefore, the set~$\big\{ \dvector{\dissection_\circ}{\delta_\bullet} + \dvector{\dissection_\circ}{\delta_\bullet} \big\} \cup \dvector{\dissection_\circ}{\dissection_\bullet \cap \dissection_\bullet'}$ contains~$|\dissection_\circ|$ vectors of~$\R^{\dissection_\circ}$ whose~$\gamma_\circ$-coordinate all vanish, so that it admits a linear dependence.\\[-.3cm]

\item Otherwise, we can assume that $\gamma_\bullet$ crosses~$\mu_\bullet$ but not~$\nu_\bullet$. Then~$\gamma_\bullet$ has a common endpoint with~$\nu_\bullet$ and~$\delta_\bullet$ (or~$\delta_\bullet'$, but we then permute notations). Changing our initial choice of~$\gamma_\circ$, we can assume that no diagonal of~$\dissection_\bullet^\mi$ separates~$\gamma_\bullet$ from~$\delta_\bullet$. We now denote clockwise 
    \begin{itemize}
    \item by~$\nu_\bullet \defeq \lambda_\bullet^0, \lambda_\bullet^1, \dots, \lambda_\bullet^\ell \eqdef \delta_\bullet$ the edges of the cell~$\cell_\bullet$ of~$\dissection_\bullet$ containing~$\nu_\bullet$ and~$\delta_\bullet$,
    \item by~$\gamma_\bullet \defeq \gamma_\bullet^0, \gamma_\bullet^1, \dots, \gamma_\bullet^k$ the edges of the cell~$\cell_\bullet^\mi$ of~$\dissection_\bullet^\mi$ containing~$\gamma_\bullet$ and crossed by~$\delta_\bullet$.
    \end{itemize}
These notations are illustrated on \fref{fig:proofdvectorFan}. We still distinguish two subcases as in \fref{fig:proofdvectorFan}:
	\begin{itemize}
	\item If $\gamma_\bullet^i$ crosses~$\lambda_\bullet^i$ for all~$i$ as in \fref{fig:proofdvectorFan}\,(middle), then~$\ell = k$ and we have the linear dependence
	\[
	\qquad\qquad 2\dvector{\dissection_\circ}{\delta_\bullet} + \dvector{\dissection_\circ}{\delta_\bullet'} = \dvector{\dissection_\circ}{\mu_\bullet} + \sum_{i \in [\ell-1]} (-1)^{(i-1)}\dvector{\dissection_\circ}{\lambda_\bullet^i}.
	\]
	It is essential here that~$\ell = k$ is even. This is guarantied by the assumption that~$\dissection_\circ$ (and thus~$\dissection_\bullet^\mi$) has no even interior cell, since~$\cell_\bullet^\mi$ is an interior \mbox{cell of~$\dissection_\bullet^\mi$ of size $k$}.
	\item Otherwise, we are in a situation similar to \fref{fig:proofdvectorFan}\,(right). Considering the maximal index~$m$ such that~$\gamma_\bullet^i$ crosses~$\lambda_\bullet^i$ for all~$i \le m$, and we have the linear dependence
	\[
	\qquad\qquad \dvector{\dissection_\circ}{\delta_\bullet} + \dvector{\dissection_\circ}{\delta_\bullet'} = \dvector{\dissection_\circ}{\mu_\bullet} + \sum_{i \in [m]} (-1)^{(i-1)} \dvector{\dissection_\circ}{\lambda_\bullet^i}. \qedhere
	\]
	\end{itemize}
\end{itemize}
\begin{figure}
	\capstart
	\centerline{\includegraphics[scale=1]{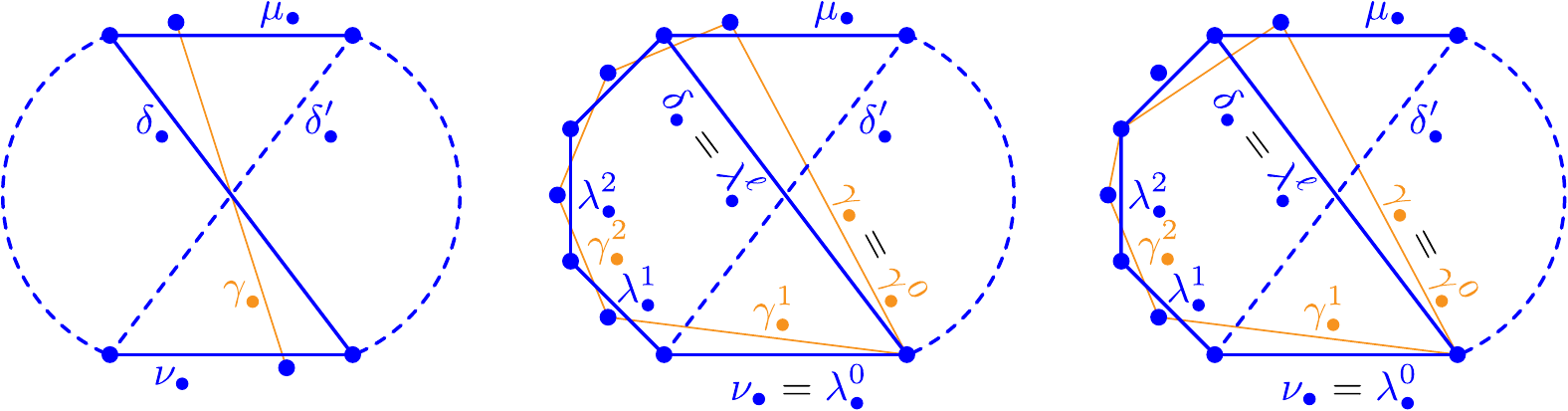}}
	\caption{Illustration of the notations and of the different cases in the proof of Theorem~\ref{thm:dvectorFan}.}
	\label{fig:proofdvectorFan}
\end{figure}
\end{proof}

\begin{example}
Following Example~\ref{exm:specialReferenceDissections}, we observe that special reference dissections give rise to the following relevant fans:
\begin{itemize}
\item For a snake triangulation~$\snake_\circ$, the $\b{d}$-vector fan~$\dvectorFan(\snake_\circ)$ coincides with the type~$A$ cluster fan of S.~Fomin and A.~Zelevinsky~\cite{FominZelevinsky-ClusterAlgebrasII}.
\item For any triangulation~$\triangulation_\circ$, the $\b{d}$-vector fan~$\dvectorFan(\triangulation_\circ)$ was already constructed in~\cite{CeballosSantosZiegler}.
\item For a quadrangulation~$\quadrangulation_\circ$ with no interior quadrangle (equivalently, with no cross), we obtain an alternative realization of the Stokes complexes studied in~\cite{Baryshnikov, Chapoton-quadrangulations}. This was observed by A.-H.~Bateni, T.~Manneville and V.~Pilaud in~\cite{BateniMannevillePilaud}.
\end{itemize}
\fref{fig:exmdvectorFans} illustrates the $\b{d}$-vector fans~$\dvectorFan(\dissection_\circ)$ for the same reference dissections~$\dissection_\circ$ as in \fref{fig:exmgvectorFans}. More precisely, we have represented the stereographic projection of the fans from the point~$[\, -1, -1, -1 \,]$. Therefore, the external face of the projection corresponds to the $\dissection_\circ$-accordion dissection~$\dissection_\bullet^\mi$. We have labeled all vertices of the projection (\ie the rays of the fan) by the corresponding $\dissection_\circ$-accordion diagonals. Compare with \fref{fig:exmgvectorFans}.

\begin{figure}
	\capstart
	\centerline{
		\begin{tabular}{l@{\hspace{1.3cm}}l}
			\includegraphics[scale=1.35]{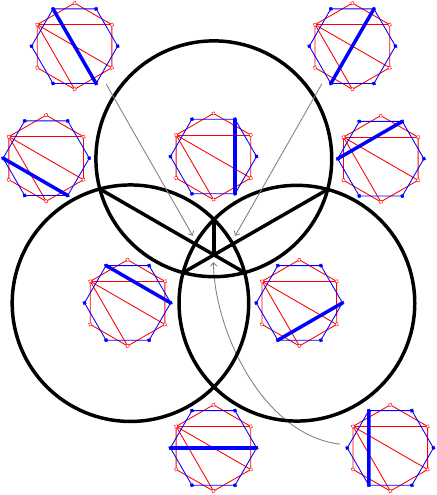} & \includegraphics[scale=1.35]{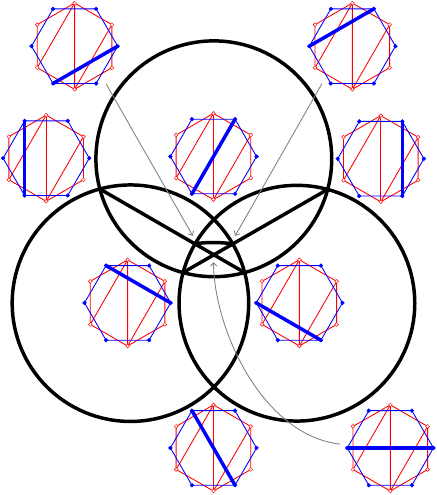} \\[.1cm]
			\includegraphics[scale=1.35]{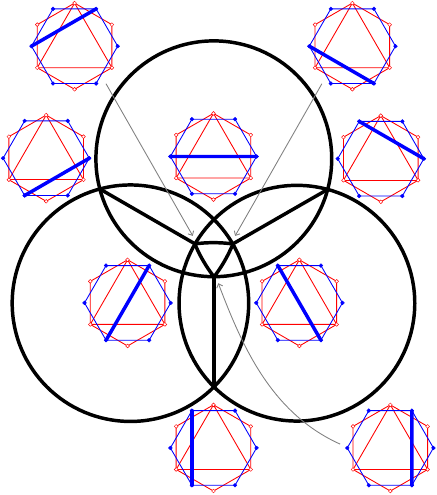} & \includegraphics[scale=1.35]{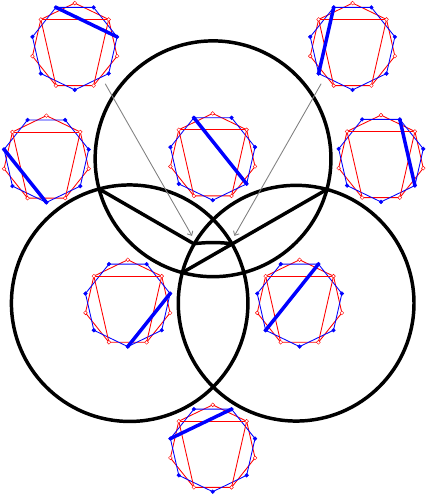}
		\end{tabular}
	}
	\caption{Stereographic projections of the $\b{d}$-vector fans~$\dvectorFan(\dissection_\circ)$ for various reference hollow dissections~$\dissection_\circ$. See \fref{fig:exmgvectorFans} for alternative simplicial fan realizations of these accordion complexes.}
	\label{fig:exmdvectorFans}
\end{figure}
\end{example}

\begin{remark}
To prove that the~$\b{d}$-vector fan~$\dvectorFan(\dissection_\circ)$ is polytopal, we would need to find suitable hyperplanes orthogonal to their rays in order to apply Theorem~\ref{theo:HohlwegLangeThomas}. For the~$\b{g}$-vector fan, these hyperplanes were defined using the height function~$\rhs{\dissection_\circ}{\delta_\bullet}$. It would be natural to use the same height function for the $\b{d}$-vector fan as well. Unfortunately, for this choice of height function, we can only prove Condition~(i) of Theorem~\ref{theo:HohlwegLangeThomas} when~$\dissection_\circ$ is a triangulation (see also~\cite{CeballosSantosZiegler}). We were not able to find suitable right hand sides for any dissection~$\dissection_\circ$.
\end{remark}

\begin{remark}
Our~$\b{d}$-vectors record the compatibility with the dissection~$\dissection_\bullet^\mi$. A priori, we could compute compatibility vectors with respect to any other maximal~$\dissection_\circ$-accordion dissection~$\dissection^\ini_\bullet$. Experiments suggest that the $\b{d}$-vector construction provides a complete simplicial fan as long as neither~$\dissection_\circ$ nor~$\dissection^\ini_\bullet$ contain no even interior cell. We checked it for reference quadrangulations with at most~$5$ diagonals. The linear dependences involved seem however much more complicated than those of the proof of Theorem~\ref{thm:dvectorFan} (in particular, they may involve $\b{d}$-vectors of diagonals not included in the cells containing~$\delta_\bullet$ and~$\delta_\bullet'$).
\end{remark}


\section{Sections and projections}
\label{sec:projection}

Recall that for a fan~$\Fan$ of~$\R^d$ and a linear subspace~$V$ of~$\R^d$, the \defn{section} of~$\Fan$ by~$V$ is the fan~$\restriction{\Fan}{V} \eqdef \set{C \cap V}{C \in \Fan}$. For a polytope~$P \subseteq \R^d$ and a projection~$\pi : \R^d \to V$, the normal fan of the projected polytope~$\pi(P)$ is the section of the normal fan of~$P$ by~$V$ \cite[Lem.~7.11]{Ziegler-polytopes}.
We now consider sections of the $\b{g}$- and $\b{d}$-vector fans by coordinate subspaces. For two dissections~${\dissection_\circ \subset \dissection_\circ'}$, we naturally identify~$\R^{\dissection_\circ}$ with the subspace spanned by~$\set{\b{e}_{\delta_\circ}}{\delta_\circ \in \dissection_\circ}$ in~$\R^{\dissection_\circ'}$.

\subsection{Coordinate sections of the $\b{d}$-vector fan}
\label{subsec:sectiondvectors}

We start by presenting sections of the $\b{d}$-vector fan which are not very surprising. The following lemma is immediate from the definition of \mbox{$\b{d}$-vectors}.

\begin{lemma}
Consider two dissections~$\dissection_\circ \subset \dissection_\circ'$, and a $\dissection_\circ'$-accordion diagonal~$\delta_\bullet$. Then we have ${\dvector{\dissection_\circ}{\delta_\bullet} \in \R^{\dissection_\circ}}$ if and only if~$\delta_\bullet$ does not cross any diagonal of~$\set{(i-1)_\bullet (j-1)_\bullet}{i_\circ j_\circ \in \dissection_\circ' \ssm \dissection_\circ}$.
\end{lemma}

\begin{corollary}
For two dissections~$\dissection_\circ \subset \dissection_\circ'$, the face complex of the section of the $\b{d}$-vector fan~$\dvectorFan(\dissection_\circ')$ by~$\R^{\dissection_\circ}$ is isomorphic to the link of the dissection~$\set{(i-1)_\bullet (j-1)_\bullet}{i_\circ j_\circ \in \dissection_\circ' \ssm \dissection_\circ}$ in the \mbox{$\dissection_\circ'$-accordion} complex~$\accordionComplex(\dissection_\circ')$. 
\end{corollary}

\subsection{Coordinate sections of the $\b{g}$-vector fan}
\label{subsec:sectiongvectors}

More relevant are the sections of the $\b{g}$-vector fan. They provide an alternative approach to polytopal realizations of the accordion complex based on projected associahedra. This approach relies on the following crucial observation.

\begin{lemma}
\label{lem:deleteDiagonals}
Consider two dissections~$\dissection_\circ \subset \dissection_\circ'$, and a $\dissection_\circ'$-accordion diagonal~$\delta_\bullet$. Then we have $\gvector{\dissection_\circ'}{\delta_\bullet} \in \R^{\dissection_\circ}$ if and only if~$\delta_\bullet$ is a $\dissection_\circ$-accordion diagonal. Moreover, in this case, the \mbox{$\b{g}$-vectors}~$\gvector{\dissection_\circ}{\delta_\bullet}$ and~$\gvector{\dissection_\circ'}{\delta_\bullet}$ coincide.
\end{lemma}

\begin{proof}
Let~$\delta_\circ \in \dissection_\circ' \ssm \dissection_\circ$. By definition, a $\dissection_\circ'$-accordion diagonal $\delta_\bullet$ does not slalom on~$\delta_\circ$ if and only if the $\delta_\circ$-coordinate of~$\gvector{\dissection_\circ}{\delta_\bullet}$ vanishes. Thus, $\delta_\bullet$ is a $\dissection_\circ$-accordion diagonal if and only if the~$\delta_\circ$-coordinate of~$\gvector{\dissection_\circ'}{\delta_\bullet}$ vanishes for all~$\delta_\circ \in \dissection_\circ' \ssm \dissection_\circ$.
\end{proof}

Based on this lemma, we obtain in the following statements an alternative realization on the $\b{g}$-vector fan, which is illustrated on \fref{fig:exmAccordiohedronProjected}.

\begin{figure}[b]
	\capstart
	\centerline{\includegraphics[scale=.28]{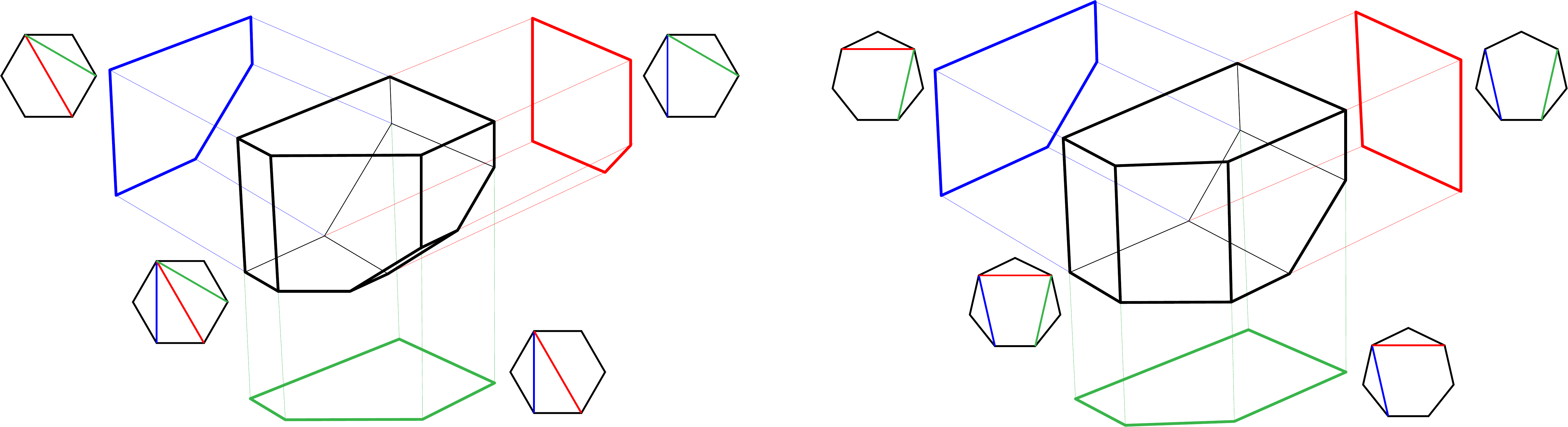}}
	\caption{Projecting accordiohedra on coordinate planes yields smaller accordiohedra.}
	\label{fig:exmAccordiohedronProjected}
\end{figure}

\begin{theorem}
\label{thm:section}
For two dissections~$\dissection_\circ \subset \dissection_\circ'$, the $\b{g}$-vector fan~$\gvectorFan(\dissection_\circ)$ is precisely the set of cones~$\set{C \in \gvectorFan(\dissection_\circ')}{C \subset \R^{\dissection_\circ}}$ and coincides with the section of the $\b{g}$-vector fan~$\gvectorFan(\dissection_\circ')$ by~$\R^{\dissection_\circ}$.
\end{theorem}


\begin{proof}
Lemma~\ref{lem:deleteDiagonals} immediately implies that~$\gvectorFan(\dissection_\circ) = \set{C \in \gvectorFan(\dissection_\circ')}{C \subset \R^{\dissection_\circ}}$. A priori, it is a subfan of the section $\restriction{\gvectorFan(\dissection_\circ')}{\R^{\dissection_\circ}} = \set{C \cap \R^{\dissection_\circ}}{C \in \gvectorFan(\dissection_\circ')}$. However, since $\gvectorFan(\dissection_\circ)$ is already a complete simplicial fan of~$\R^{\dissection_\circ}$, it coincides with~$\restriction{\gvectorFan(\dissection_\circ')}{\R^{\dissection_\circ}}$. 
\end{proof}

\begin{theorem}
\label{thm:projection}
For two dissections~$\dissection_\circ \subset \dissection_\circ'$, the $\b{g}$-vector fan~$\gvectorFan(\dissection_\circ)$ is realized by the orthogonal projection of the $\dissection_\circ'$-accordiohedron~$\Acco(\dissection_\circ')$ on~$\R^{\dissection_\circ}$, which is equivalently described~by:
\begin{itemize}
\item the convex hull of the points~$\sum_{\delta_\bullet \in \dissection_\bullet} \rhs{\dissection_\circ'}{\delta_\bullet} \cdot \cvector{\dissection_\circ}{\dissection_\bullet}{\delta_\bullet}$ for all~$\dissection_\circ$-accordion dissections~$\dissection_\bullet$,
\item the intersection of the half-spaces~$\set{\b{x} \in \R^{\dissection_\circ}}{\dotprod{\gvector{\dissection_\circ}{\gamma_\bullet}}{\b{x}} \le \rhs{\dissection_\circ'}{\delta_\circ}}$ for all \mbox{$\dissection_\circ$-accor}\-dion diagonals~$\gamma_\bullet$.
\end{itemize}
\end{theorem}

\begin{proof}
Since~$\gvectorFan(\dissection_\circ')$ is the normal fan of~$\Acco(\dissection_\circ')$, Theorem~\ref{thm:section} implies that~$\gvectorFan(\dissection_\circ) = \restriction{\gvectorFan(\dissection_\circ')}{\R^{\dissection_\circ}}$ is the normal fan of the orthogonal projection of~$\Acco(\dissection_\circ')$ on~$\R^{\dissection_\circ}$ \cite[Lem.~7.11]{Ziegler-polytopes}. We therefore just need to prove the given vertex and facet descriptions of this projection. First, since $\gvectorFan(\dissection_\circ) = \restriction{\gvectorFan(\dissection_\circ')}{\R^{\dissection_\circ}}$, the inequalities of the projection of~$\Acco(\dissection_\circ')$ on~$\R^{\dissection_\circ}$ are just the inequalities of~$\Acco(\dissection_\circ')$ whose normal vectors are in~$\R^{\dissection_\circ}$. Finally, the vertex description follows from the inequality description using the same argument as in Lemma~\ref{lem:pointInHyp}.
\end{proof}

\begin{remark}
The projection of the accordiohedron~$\Acco(\dissection_\circ')$ on~$\R^{\dissection_\circ}$ differs from the accordiohedron~$\Acco(\dissection_\circ)$: they have both~$\gvectorFan(\dissection_\circ)$ as normal fan, but their precise geometry is different.
\end{remark}

\begin{corollary}
\label{coro:projectionAssociahedron}
For any hollow dissection~$\dissection_\circ$, the $\b{g}$-vector fan~$\gvectorFan(\dissection_\circ)$ is realized by a projection of an associahedron of~\cite{HohlwegPilaudStella}.
\end{corollary}

\begin{proof}
Apply Theorem~\ref{thm:projection} to any triangulation~$\triangulation_\circ$ that refines~$\dissection_\circ$.
\end{proof}

\begin{remark}
\label{rem:simplerProofs}
Approaching accordion complexes as coordinate sections of $\b{g}$-vector fans actually provides more concise (but also less instructive) proofs for Sections~\ref{subsec:pseudomanifold} and~\ref{subsec:gvectorFan}. Namely, consider any dissection~$\dissection_\circ$ and let~$\triangulation_\circ$ be a triangulation that refines~$\dissection_\circ$. 
The sign coherence property for triangulations (see Corollary~\ref{coro:signCoherence}) shows that the section $\restriction{\gvectorFan(\triangulation_\circ)}{\R^{\dissection_\circ}} = \set{C \cap \R^{\dissection_\circ}}{C \in \gvectorFan(\triangulation_\circ)}$ actually coincides with~$\set{C \in \gvectorFan(\triangulation_\circ)}{C \subset \R^{\dissection_\circ}}$. Therefore, this gives an alternative concise proof that the collection of cones~$\set{C \in \gvectorFan(\triangulation_\circ)}{C \subset \R^{\dissection_\circ}}$ forms a complete simplicial fan. Moreover, this fan has the same combinatorics as the $\dissection_\circ$-accordion complex~$\accordionComplex(\dissection_\circ)$ by Lemma~\ref{lem:deleteDiagonals}. We conclude directly that~$\accordionComplex(\dissection_\circ)$ is a pseudomanifold realized by the fan~$\set{C \in \gvectorFan(\triangulation_\circ)}{C \subset \R^{\dissection_\circ}}$ and by the orthogonal projection of the associahedron~$\Asso(\triangulation_\circ)$ on~$\R^{\dissection_\circ}$.
\end{remark}

\subsection{Cluster algebra analogues}
\label{subsec:clusterAlgebras}

The perspective on accordion complexes developed in this section also opens the door to generalizations on arbitrary cluster algebras (finite type or not). Namely, consider an arbitrary cluster~$X_\circ = (x_\circ^1, \dots, x_\circ^m)$ in an arbitrary cluster algebra~$\cA$. For any cluster variable~$y \in \cA$, we denote by~$\gvector{X_\circ}{y} \in \R^m$ and~$\dvector{X_\circ}{y} \in \R^m$ the $\b{g}$- and~$\b{d}$-vectors of~$y$ computed with respect to~$X_\circ$, see~\cite{FominZelevinsky-ClusterAlgebrasI, FominZelevinsky-ClusterAlgebrasIV}. Fix a non-empty proper subset~$I$ of~$[m]$. We consider two natural subcomplexes of the cluster complex of~$\cA$:
\begin{itemize}
\item the subcomplex~$\restrictedComplex{d}{X_\circ}{I}$ induced by the variables~$y$ such that~$\dvector{X_\circ}{y}_i = 0$ for all~$i \in I$,
\item the subcomplex~$\restrictedComplex{g}{X_\circ}{I}$ induced by the variables~$y$ such that~$\gvector{X_\circ}{y}_i = 0$ for all~$i \in I$.
\end{itemize}
It is well-known that the subcomplex~$\restrictedComplex{d}{X_\circ}{I}$ is the cluster complex obtained by freezing all variables~$x_i$ for~$i \in I$. For example in type~$A$, it is a join of simplicial associahedra and it can therefore be realized by a product of smaller associahedra. In contrast, we do not know whether the subcomplex~$\restrictedComplex{g}{X_\circ}{I}$ has been investigated. The present paper dealt with the type~$A$~situation.

\begin{example}
\label{exm:clusterAlgebra}
Let~$\triangulation_\circ$ be a triangulation, with internal diagonals labeled by~$1, \dots, m$. Consider the corresponding type~$A_m$ cluster~$X_\circ$. Then for any non-empty proper subset~$I$ of~$[m]$, the subcomplex~$\restrictedComplex{g}{X_\circ}{I}$ is isomorphic to the $\dissection_\circ$-accordion complex, where~$\dissection_\circ$ is the dissection obtained by deleting in~$\triangulation_\circ$ the diagonals labeled by~$I$.
\end{example}

\begin{example}
Example~\ref{exm:clusterAlgebra} extends to cluster algebras on surfaces~\cite{FominShapiroThurston, FominThurston}, using accordions of dissections of surfaces.
\end{example}

The following statement extends Theorem~\ref{thm:section} to arbitrary cluster algebras.

\begin{theorem}
The subset~$\set{C \in \gvectorFan(X_\circ)}{C \subseteq \R^{[m] \ssm I}}$ of the $\b{g}$-vector fan~$\gvectorFan(X_\circ)$ of~$X_\circ$ coincides with the section~$\restriction{\gvectorFan(X_\circ)}{\R^{[m] \ssm I}} = \set{C \cap \R^{[m] \ssm I}}{C \in \gvectorFan(X_\circ)}$.
\end{theorem}

\begin{proof}
The inclusion~$\set{C \in \gvectorFan(X_\circ)}{C \subseteq \R^{[m] \ssm I}} \subseteq \restriction{\gvectorFan(X_\circ)}{\R^{[m] \ssm I}}$ is clear. For the reverse inclusion, we use the sign coherence property of $\b{g}$-vectors in cluster algebras, which was conjectured in~\cite[Conj.~6.13]{FominZelevinsky-ClusterAlgebrasIV} and proved in~\cite[Thm.~5.1]{GrossHackingKeelKontsevich} in general. This property implies that the coordinate plane~$\R^{[m] \ssm I}$ intersects any cone~$C$ of~$\gvectorFan(X_\circ)$ in a face~$C'$. This shows that~$C \cap \R^{[m] \ssm I} = C'$ belongs to~$\set{C \in \gvectorFan(X_\circ)}{C \subseteq \R^{[m] \ssm I}}$.
\end{proof}

\begin{corollary}
The subcomplex~$\restrictedComplex{g}{X_\circ}{I}$ induced by the variables~$y$ such that~$\gvector{X_\circ}{y}_i = 0$ for all~$i \in I$ is a pseudomanifold.
\end{corollary}

Moreover, extending the result of C.~Hohlweg, C.~Lange and H.~Thomas~\cite{HohlwegLangeThomas} in the acyclic case, C.~Hohlweg, V.~Pilaud and S.~Stella recently constructed a polytope~$\Asso(X_\circ)$ realizing the $\b{g}$-vector fan~$\gvectorFan(X_\circ)$ in~\cite{HohlwegPilaudStella}. We can use this associahedron to realize  the subcomplex~$\restrictedComplex{g}{X_\circ}{I}$ as a convex polytope, extending Theorem~\ref{thm:projection}.

\begin{corollary}
The orthogonal projection of~$\Asso(X_\circ)$ on~$\R^{[m] \ssm I}$ is a realization of~$\restrictedComplex{g}{X_\circ}{I}$.
\end{corollary}

%


Finally, when oriented in the suitable direction~$v$ (the sum of the positive roots, or equivalently the sum of the fundamental weights), the graph of the generalized associahedron~$\Asso(X_\circ)$ is the Hasse diagram of a Cambrian lattice~\cite{Reading-CambrianLattices}. One can similarly orient the graph of the projection of~$\Asso(X_\circ)$ on~$\R^{[m] \ssm I}$ in the direction of the projection of~$v$ on~$\R^{[m] \ssm I}$. Is the resulting graph the Hasse diagram of a lattice? Combining the results of~\cite{GarverMcConville} with that of the present paper shows that this property holds in type~$A$. We also computationally verified the statement in types~$B_4$, $B_5$, $D_4$ and~$D_5$. Following~\cite{GarverMcConville} it seems promising to construct first a lattice structure on biclosed sets of $\b{c}$-vectors, and to obtain then the graph of the projection of~$\Asso(X_\circ)$ on~$\R^{[m] \ssm I}$ as the Hasse diagram of a lattice quotient.

To conclude, let us mention that the ideas developed in this section have also inspired further investigation of sections of $\b{g}$-vector fans of support $\tau$-tilting complexes of associative algebras, see~\cite{PilaudPlamondonStella} and~\cite[Sect.~4.2.6]{PaluPilaudPlamondon}.


\section*{Acknoledgements}

We thank C.~Hohlweg and S.~Stella for many helpful discussions on realizations of the associahedron~\cite{HohlwegPilaudStella} which were the starting point of this paper. We are grateful to F.~Chapoton for various conversations on quadrangulations and Stokes posets, and to A.~Garver and \mbox{T.~McConville} for introducing us with the accordion complexes during FPSAC'16. Their works~\cite{Chapoton-quadrangulations, GarverMcConville} were highly inspiring and motivating. We also thank N.~Thiery for a question which led to the approach of Section~\ref{subsec:sectiongvectors}, and to P.-G.~Plamondon for discussions on the generalization to cluster algebras presented in Section~\ref{subsec:clusterAlgebras}. Finally, we are grateful to two anonymous referees for their attentive reading and their suggestions on the content and presentation which largely improved our original draft.


\bibliographystyle{alpha}
\bibliography{accordionComplex}
\label{sec:biblio}

\end{document}